\documentclass[11pt]{amsart}
\usepackage{geometry}                
\usepackage{graphicx}
\usepackage{amssymb}
\usepackage{epstopdf}
\usepackage[mathscr]{euscript}
\title{AK-ASD Twistor}

\begin{document}
\title{Almost-K\"ahler anti-self-dual metrics on $K3\#3\overline{\mathbb{CP}_{2}}$}

\author{Inyoung Kim}
\maketitle
 
\begin{abstract}
  Donaldson-Friedman constructed anti-self-dual classes on $K3\#3\overline{\mathbb{CP}_{2}}$ using twistor space.
  We show that some of these conformal classes have almost-K\"ahler representatives.

  \end{abstract}
\maketitle

\section{\large\textbf{Introduction}}\label{S:Intro}

On a smooth, oriented riemannian 4-manifold $(M, g)$, 2-forms decomposes as self-dual and anti-self-dual 2-forms $\Lambda^{2}=\Lambda^{+}\oplus \Lambda^{-}$,
 according to the eigenvalue of the Hodge star operator $*$. By definition, a 2-form $\alpha$ is called self-dual if $*\alpha=\alpha$. 
 Then the curvature operator takes the form according to this decomposition of 2-forms $\Lambda^{2}=\Lambda^{+}\oplus\Lambda^{-}$,

 \[R=
 \begin{pmatrix} W_{+}+\frac{s}{12}&\dot{r}\\
 \dot{r}&W_{-}+\frac{s}{12}\\
 \end{pmatrix},\]
 where $\dot{r}$ comes from the trace-free Ricci curvature. 
 If $W_{+}=0$, then $g$ is called to be an anti-self-dual metric.

Let $(M, \omega)$ be a 4-dimensional symplectic manifold. 
The space of almost-complex structures which are compatible with the symplectic form, 
$\omega(v, w)=\omega(Jv, Jw)$ is nonempty and contractible [23].
If we define $g(v, w):=\omega(x, Jy)$, then $g$ is a metric which is compatible with $J$, $g(v, w)=g(Jv, Jw)$. 
We call such a metric $g$ an almost-K\"ahler metric. 
Note that $\omega$ is a self-dual harmonic 2-form of length $\sqrt{2}$ with respect to $g$. 
On the other hand, by conformal invariant properties, if we have an anti-self-dual metric $g$ and a nondegenerate self-dual harmonic 2-form, 
then, there exists a unique almost-K\"ahler anti-self-dual metric  in the conformal class of $g$ [5].

\vspace{20pt}

Let $(M, g)$ be an oriented, smooth, compact Riemannian 4-manifold.
Then there is Weitzenb\"ock formula for a self-dual 2-form $\omega$, 
\[\Delta\omega=\nabla^{*}\nabla\omega-2W_{+}(\omega, \cdot)+\frac{s}{3}\omega,\]
where $s$ is the scalar curvature. 

Let $(M, g, \omega)$ be an almost-K\"ahler anti-self-dual 4-manifold. 
Then $\omega$ is a self-dual harmonic 2-form of length $\sqrt{2}$, we get 
\[0=|\nabla\omega|^{2}+\frac{2s}{3}.\]

Thus, $s\leq0$ in case of almost-K\"ahler anti-self-dual metrics.
Moreover, $s\equiv 0$ if and only if $(g, J, \omega)$ is a K\"ahler manifold.
If an almost-K\"ahler anti-self-dual metric is not K\"ahler, we call it a strictly almost-K\"ahler anti-self-dual metric.
Let $(M, g, J)$ be a K\"ahler manifold with $\int_{M}sd\mu_{g}\geq0$.
Then either the first Chern class $c_{1}^{\mathbb{R}}\in H^{2}(M, \mathbb{R})=0$ or its Kodaira dimension is $-\infty$ [29]. 
Using the formula 
\[c_{1}^{2}=(2\chi+3\tau)(M)=\frac{1}{4\pi^{2}}\int_{M}(\frac{s^{2}}{24}+2|W_{+}|^{2}-\frac{|\dot{r}|^{2}}{2})d\mu_{g},\]
if $\mathbb{CP}_{2}\#n\overline{\mathbb{CP}_{2}}$ admit scalar-flat K\"ahler metrics, then $n\geq 10$. 
In [14], it was shown that there exist strictly almost-K\"ahler anti-self-dual metrics by deforming scalar-flat K\"ahler metrics on certain manifolds. 
Using the Seiberg-Witten invariant, it was shown that if $\mathbb{CP}_{2}\#n\overline{\mathbb{CP}_{2}}$ admits an almost-K\"ahler anti-self-dual metric, then $n\geq10$ [14].

In this respect, it might be an  interesting question whether there exists a manifold which admits almost-K\"ahler anti-self-dual metrics but not scalar-flat K\"ahler metrics. 
$K3\#n\overline{\mathbb{CP}_{2}}$ for $n\geq 3$ are candidates 
since they do not admit scalar-flat K\"ahler metrics [29] but it was shown that there exists anti-self-dual metrics on them [6]. 
In this paper, we show that some of anti-self-dual conformal classes constructed by Donaldson-Friedman in [6] are almost-K\"ahler,
using twistor interpretation of self-dual harmonic 2-forms. 

\vspace{20pt}

\newtheorem{Theorem}{Theorem}
\begin{Theorem}
There exist strictly almost-K\"ahler anti-self-dual metrics on $K3\#n\overline{\mathbb{CP}_{2}}$ for $n\geq 3$. 
\end{Theorem}

\vspace{20pt}

$\mathbf{Acknowledgement}$ : The author is most grateful to Prof. Claude LeBrun for suggesting this problem
and precious advices. The author is thankful to Jongsu Kim  for helpful discussions. The author would like to thank Eui-Sung Park for helpful comment regarding Lemma 5.
The author is very thankful to Chanyoung Sung and Korea National University of Education for supports and opportunities, by which 
this work has been able to be carried out. 
This article is supported by NRF-2018R1D1A3B07043346.

\vspace{50pt}

\section{\large\textbf{Twistor spaces}}\label{S:Intro}
An oriented, riemannian 4-manifold $(M, g)$  with an anti-self-dual metric(ASD) corresponds to the complex 3-manifold, which is called the twistor space [2, 24]. 
Consider the unit sphere bundle of self-dual 2-form $p: \mathbf{S}(\Lambda^{+})\to M$. 
 Using the Levi-Civita connection, we can split the tangent bundle of $Z:=\mathbf{S}(\Lambda^{+})$ by
\[T_{z}(Z)=V_{z}\oplus (p^{*}TM)_{z}.\]
On $V_{z}$, which is the tangent space of the fiber, we define $J_{V}$ be the $-90^{\circ}$ rotation and on $(p^{*}TM)_{z}$, we put the almost-complex structure determined by $z$. 
The fundamental theorem by Penrose and Atyiah, Hitchin, Singer is that this complex structure is integrable if $g$ is anti-self-dual [2, 24]. 
We note that the twistor space $Z$ can be also given by $P(\mathbf{V}_{+})$, where $\mathbf{V}_{+}$ is the positive spinor bundle [12]. 
Moreover, locally there exists the bundle $\mathbf{H}$ such that $\mathbf{H}$ is the Hopf bundle on each fiber $\mathbb{CP}_{1}$ 
and $\mathbf{H}^{2}$ exists globally [12]. 
Then it is shown that the canonical line bundle $K$ of $Z$ is isomorphic to $\mathbf{H}^{-4}$, which we denote by $\mathcal{O}(-4)$ [2].
Also, there is a fixed-point free anti-holomorphic involution $\sigma$, defined by the quaternionic structure on $\mathbf{V}_{+}$, $\sigma^{2}=Id$ [2]. 
$\sigma$ is the antipodal map on each fiber and $\sigma$ preserves each twistor line. 
We call $\sigma$ be the real structure on $Z$. 
Then $\sigma$ induces a complex-anti-linear map on $H^{1}(Z, K)$. 
We call an element of $H^{1}(Z, K)$ is real if it is invariant under $\sigma$.

Conversely, let $Z$ be a complex 3-manifold which has a fixed-point free anti-holomorphic involution $\sigma$ such that $\sigma^{2}=Id$.
Suppose further $Z$ is fibered by $\sigma$-invariant holomorphic curves $\mathbb{CP}_{1}$,
which are called the real twistor lines and normal bundle of each real twistor line is isomorphic to $\mathcal{O}(1)\oplus \mathcal{O}(1)$. 
Then there is a corresponding 4-manifold with the anti-self-dual metric [2, 24].

\vspace{20pt}
Let $(M_{i}, g_{i})$ be anti-self-dual 4-manifolds and let $Z_{i}$ be twistor spaces corresponding to $M_{i}$. 
Take a twistor line $l_{i}\subset Z_{i}$ and by blowing up this line, 
we get an exceptional divisor $Q_{i}=\mathbb{CP}_{1}\times\mathbb{CP}_{1}$ on $Z_{i}$ and 
we denote blown up manifolds by $\tilde{Z_{i}}$.
We identify $\tilde{Z}_{1}$ and $\tilde{Z}_{2}$ along $Q_{i}$ by interchanging factors $Q_{1}$ and $Q_{2}$
and we denote the identified singular manifold with normal crossing divisor $Q$ by $Z_{0}$.
A real structure $\sigma_{i}$ on $Z_{i}$ extends to $\tilde{Z}_{i}$ such that $\tilde{\sigma}_{i}|_{Z_{i}}=\sigma_{i}$ and therefore induces the real structure $\sigma_{0}$ on $Z_{0}$. 
It was shown in [6] that  if $H^{2}(Z_{i}, \Theta_{Z_{i}})=0$, 
then there exists a complex deformation of $Z_{0}$
and this deformation produces anti-self-dual metrics on $M_{1}\#M_{2}$. 
From this, it was shown that $n\overline{\mathbb{CP}_{2}}$ admit anti-self-dual metrics [6]. 

In this paper, we are in particular interested in the existence of almost-K\"ahler anti-self-dual metrics on $K_{3}\#3\overline{\mathbb{CP}_{2}}$, more generally $K_{3}\#n\overline{\mathbb{CP}_{2}}$, $n\geq 3$.
Ricci-flat K\"ahler metrics are  anti-self-dual. K3 surface is known to admit such metrics [30] and 
we consider the corresponding twistor space $Z$. Note that $H^{2}(Z, \Theta_{Z})\neq 0$ [6]. 
However, by overcoming the obstruction, it was shown that $NK_{3}\#n\overline{\mathbb{CP}_{2}}$
admit anti-self-dual metrics for $N>0$ and $n\geq 2N+1$ [6]. 

\begin{Theorem}
(Donaldson-Friedman) There exist anti-self-dual metrics on $NK_{3}\#n\overline{\mathbb{CP}_{2}}$
for $N>0$ and $n\geq 2N+1$ [6].

\end{Theorem}

This  method was developed further by LeBrun and Singer [21] when a 4-manifold $M$ admits an isometric $\mathbb{Z}_{2}$-action with $k$-isolated fixed points. 
Moreover, by considering cohomological interpretation of positive scalar curvature condition [1], it was shown in [19] that 
$\tilde{X}\#n\overline{\mathbb{CP}_{2}}$ admit an anti-self-dual metric of positive scalar curvature if $M$ does and $H^{2}(Z, \Theta_{Z})=0$,
where $Z$ is the twistor space of $M$.
Here $X=M/\mathbb{Z}_{2}$ and $\tilde{X}$ be the oriented manifold by replacing singularity of $M/\mathbb{Z}_{2}$
by 2-sphere of self-intersection $-2$. 
Also, Kalafat showed that if $(M_{i}, g_{i})$ are anti-self-dual 4-manifolds of positive scalar curvature such that
their twistor spaces $Z_{i}$ satisfy $H^{2}(Z_{i}, \Theta_{Z_{i}})=0$, then $M_{1}\#M_{2}$ admits an anti-self-dual metric of positive scalar curvature [13]. 
In this paper, we use this method of LeBrun in the construction of almost-K\"ahler ASD metrics on $K3\#n\overline{\mathbb{CP}_{2}}$ for $n\geq 3$.

\vspace{20pt}

Let $M$ be an anti-self-dual space and let $Z$ be its twistor space. 
It was shown there are correspondences between certain cohomology groups on a twistor space $Z$
and solutions of differential equations on a 4-manifold $M$ [7], [11]. 
One of these correspondences we need in this paper is the following. 
Consider $Spin(4)=SU(2)\times SU(2)$. 
Let $\mathbf{V}_{\pm}$ be the basic Spin representations of two factors. 
We denote $\mathbf{S}_{+}^{m}$ for $\mathbf{S}^{m}\mathbf{V}_{+}$, where $\mathbf{S}^{m}$ denote the symmetric power
and $\mathbf{S}_{-}^{m}$ for $\mathbf{S}^{m}\mathbf{V}_{-}$.
By following [11], 
we note the following operator,
\[\mathbf{D}_{m}:\Gamma(\mathbf{S}_{+}^{m})\to \Gamma(\mathbf{S}^{m-1}_{+}\otimes \mathbf{S}_{-})\]
with weight $\frac{1}{2}(m+2)$.

\begin{Theorem}
[7,11] Let $M$ be an anti-self-dual space and $Z$ be its twistor space. 
Then 
\[T: H^{1}(Z, \mathcal{O}(-m-2))\to \Gamma(\mathbf{S}^{m}_{+}) \]
defines an isomorphism onto the space of solutions to $\mathbf{D}_{m}\phi=0$, for $m\geq 0$. 
In particular, when $m=2$, a real element of $H^{1}(Z, \mathcal{O}(-4))$ corresponds to a real self-dual closed 2-form. 
\end{Theorem}

Let $(X, \mathcal{O}_{X}), (Y, \mathcal{O}_{Y})$ be complex spaces.
A map between complex spaces $X, Y$ is given by $(f, f^{\#})$, where $f:X\to Y$ and $f^{\#}:\mathcal{O}_{Y}\to f_{*}\mathcal{O}_{X}$.
Here $f_{*}\mathcal{O}_{X}$ is the direct image sheaf. 
Then a map $(f, f^{\#}):(X, \mathcal{O}_{X})\to(Y, \mathcal{O}_{Y})$ is called flat over $y$ if $\mathcal{O}_{X, x}$ is a flat module over $\mathcal{O}_{Y, f(x)}$
via the map $f^{\#}:\mathcal{O}_{Y, f(x)}\to \mathcal{O}_{X, x}$.
A sheaf of $\mathcal{O}_{X}$-module $\mathcal{F}$ on $X$ is said to be flat if $\mathcal{F}_{x}$ is flat module over $\mathcal{O}_{Y, f(x)}$.

If $\mathcal{G}$ be a sheaf of $\mathcal{O}_{Y}$-module, then $f^{-1}\mathcal{G}$ is a $f^{-1}\mathcal{O}_{Y}$-module. Moreover, 
via the map $f^{-1}\mathcal{O}_{Y}\to \mathcal{O}_{X}$, $f^{*}\mathcal{G}$ is defined to be $f^{-1}\mathcal{G}\otimes_{f^{-1}\mathcal{O}_{Y}}\mathcal{O}_{X}$. 
Then $f^{*}\mathcal{G}$ is an $\mathcal{O}_{X}$-module. 

Let $X\subset Y$ and $i:X\to Y$ be the inclusion map. For a sheaf of $\mathcal{O}_{Y}$-module $\mathcal{G}$, 
we consider the following map 
\[r_{X}:H^{k}(Y, \mathcal{G})\to H^{k}(X, i^{*}\mathcal{G}).\]
We denote $\alpha|_{X}:=r_{X}(\alpha)$, where $\alpha\in H^{i}(Y, \mathcal{G})$.

Let $\alpha\in H^{1}(Z, K)$. 
Then, $\alpha|_{\mathbb{CP}_{1}}\in H^{1}(\mathbb{CP}_{1}, \mathcal{O}_{\mathbb{CP}_{1}}(-4))$. 
By Serre duality, 
\[H^{1}(\mathbb{CP}_{1}, \mathcal{O}_{\mathbb{CP}_{1}}(-4))=H^{0}(\mathbb{CP}_{1}, \mathcal{O}_{\mathbb{CP}_{1}}(2)). \]
Then $\phi\in H^{0}(\mathbb{CP}_{1}, \mathcal{O}_{\mathbb{CP}_{1}}(2))$ correspond to $\mathbf{S}_{+}^{2}$. 
Since $\mathbb{CP}_{1}=P((\mathbf{S}^{*}_{+})_{x}\setminus0)$, a section $\phi$ gives rise to a homogenous polynomial of degree 2 on $(\mathbf{S}^{*}_{+})_{x}\setminus0$,
which gives an element of $(\mathbf{S}_{+}^{2})_{x}$ [11]. 
When $m=2$, we have $\mathbf{S}^{2}_{+}=\Gamma(\Lambda^{+}_{c})$ and $\mathbf{D}_{2}$ on $\Gamma(\Lambda^{+})$ is the exterior derivative $d$ [11]. 
 Thus, an element of $H^{1}(Z, \mathcal{O}_{Z}(K))$ corresponds to a self-dual closed 2-form. 
 Since $\Delta=dd^{*}+d^{*}d$ and $d^{*}=*d*$ for a 2-form on a 4-manifold, a self-dual closed 2-form is in particular harmonic. 
A real element of $H^{1}(Z, \mathcal{O}_{Z}(K))$ corresponds to a self-dual closed real 2-form on a 4-manifold.

\newtheorem{Remark}{Remark}
\begin{Remark}
Let $Z_{t}$ be the twistor space of $(K3\#3\overline{\mathbb{CP}_{2}}, g_{t})$, where $g_{t}$ is the family of anti-self-dual metrics 
constructed by Donaldson-Friedman. 
If we have a real cohomology class $\alpha\in H^{1}(Z_{t}, \mathcal{O}_{Z_{t}}(K_{t}))$ such that $\alpha|_{l}\neq 0$ for any twistor line $l\in Z_{t}$, 
we get a nondegenerate real self-dual closed 2-form. In particular, this gives an almost-K\"ahler anti-self-dual metric on $K3\#3\overline{\mathbb{CP}_{2}}$.
\end{Remark}

\vspace{50pt}

\section{\large\textbf{Extension of a Cohomology class}}\label{S:Intro}
Let $(X, \mathcal{O}_{X}), (Y, \mathcal{O}_{Y})$ be complex spaces
with maps $(f, f^{\#})$, where $f:X\to Y$ and $f^{\#}:\mathcal{O}_{Y}\to f_{*}\mathcal{O}_{X}$.
Let $\mathcal{F}$ be a sheaf of $\mathcal{O}_{X}$-module on $X$. 
Let us define higher direct image sheaf $R^{i}f_{*}(\mathcal{F})$. This is the sheaf associated to the following presheaf on $Y$,
\[V\longmapsto H^{i}(f^{-1}(V), \mathcal{F}|_{f^{-1}(V)}).\]

\begin{Theorem}
[3, 4] Let $X, Y$ be reduced complex spaces. 
Suppose $f:X\to Y$ be a proper morphism and $\mathcal{F}$ be a coherent sheaf on $X$, which is flat over $y$ for all $y\in Y$. 
Let $\mathcal{I}_{y}$ be the ideal sheaf of $y$. 
Then we have 

(1) For all $q\geq0$, $h^{q}(X_{y}, \mathcal{F}_{y})$ is an upper semi-continuous function of $y$. 

(2) $R^{q}f_{*}(\mathcal{F})$ is locally free if $h^{q}(X_{y}, \mathcal{F}_{y})$ is constant.

(3) If $h^{q}(X_{y}, \mathcal{F}_{y})$ is constant, then the map $(R^{q}f_{*}\mathcal{F})_{y}/\mathcal{I}_{y}(R^{q}f_{*}\mathcal{F})_{y}\to H^{q}(X_{y}, \mathcal{F}_{y})$ 
is bijective.

\end{Theorem}
In Theorem 4, on $f^{-1}(y)$, we consider the inclusion map 
$i_{y}:f^{-1}(y)\to X$ and we define $\mathcal{F}_{y}:=i_{y}^{*}\mathcal{F}$.

In this paper, we are in particular interested in the case $K3\#3\overline{\mathbb{CP}_{2}}$.
Let $Z$ be the twistor space of $K3$ with a Ricci-flat K\"ahler metric  and let $\overline{\mathbb{CP}_{2}}$ be $\mathbb{CP}_{2}$
with the non-standard orientation and $g_{FS}$ be the Fubini-Study metric. 
Let $F_{i}$ be the twistor space of $(\overline{\mathbb{CP}_{2}}, g_{FS})$. 
Then take a twistor line $l_{i}$, $i=1, 2, 3$ on $Z$ and by blowing up $l_{i}$, 
we get $\tilde{Z}$ with an exceptional divisor $Q'_{i}$ for $i=1, 2, 3$, which is $\mathbb{CP}_{1}\times\mathbb{CP}_{1}$. 
 The normal bundle of $Q'_{i}$ in $\tilde{Z}$ is $\mathcal{O}(1, -1)$.  
 In this paper, $\mathcal{O}(a)$ means some power of tautological line bundle on $\mathbb{CP}_{n}$
 or the sheaf of sections of it according to the context. 
Let $\pi_{1}$ be the projection of $\mathbb{CP}_{1}\times\mathbb{CP}_{1}$ to the first factor
 and $\pi_{2}$ the second factor. 
By $\mathcal{O}(a, b)$,  we mean $\pi_{1}^{*}\mathcal{O}(a)\otimes \pi_{2}\mathcal{O}(b)$. 

Similarly, by blowing up a twistor line, we get $(\tilde{F}_{i}, Q''_{i})$ such that the normal bundle of $Q''_{i}$ in $\tilde{F}_{i}$ is $\mathcal{O}(1, -1)$. 
We identify $\tilde{Z}$ with $\tilde{F}_{i}$ by identifying $Q'_{i}\subset \tilde{Z}$ and $Q''_{i}\subset \tilde{F}_{i}$ by switching each factor in the $Q'_{i}$ and $Q''_{i}$
and denote the identification of $Q'_{i}$ and $Q''_{i}$ by $Q_{i}$.
Let $\tilde{\mathscr{F}}=\tilde{F}_{1}\cup \tilde{F}_{2}\cup \tilde{F}_{3}$ and $\mathscr{Q}=Q_{1}\cup Q_{2}\cup Q_{3}$. 
Then we get a singular space $Z_{0}=\tilde{Z}\cup _{\mathscr{Q}}\tilde{\mathscr{F}}$ with normal crossing divisor $\mathscr{Q}$.

 In this paper, we consider the following type of deformation, as suggested by LeBrun in [19]. 
A 1-parameter family of standard deformation of a singular complex space $Z_{0}$
is a flat, proper, holomorphic map $\varpi:\mathcal{Z}\to\mathcal{U}$ with an anti-holomorphic involution $\sigma:\mathcal{Z}\to\mathcal{Z}$ such that 
$\sigma|_{Z_{0}}=\sigma_{0}$ and $\mathcal{U}\subset\mathbb{C}$ is an open neighborhood of $0$. 
$\mathcal{Z}$ is a complex 4-manifold and when $u\in \mathcal{U}$ is non-zero real, $Z_{\mathcal{U}}=\varpi^{-1}(u)$ is a twistor space.
For a precise definition, we refer to [19]. 
By the construction of anti-self-dual metrics in [6], there exists a standard deformation of $Z_{0}$.
We denote this standard deformation by 
\[\varpi:\mathcal{Z}\to \mathcal{U},\]
with fiber $Z_{t}$ which is a smoothing of a singular twistor space $Z_{0}$ and 
$\mathcal{U}\in\mathbb{C}$ is a neighborhood of the origin. 

Let $K_{\mathcal{Z}}$ be the canonical bundle of $\mathcal{Z}$ and 
let $\mathcal{I}_{\tilde{\mathscr{F}}}$ be the ideal sheaf of $\tilde{\mathscr{F}}\in \mathcal{Z}$.
Then we consider the invertible sheaf $\mathcal{O}_{\mathcal{Z}}(K_{\mathcal{Z}})\otimes2\mathcal{I}_{\tilde{\mathscr{F}}}$. 
We use $K$ instead of  $K_{\mathcal{Z}}$. 
We apply Leray spectral sequence to this map $\varpi$.

Since $\mathcal{O}_{\mathcal{Z}}(K)\otimes2\mathcal{I}_{\tilde{\mathscr{F}}}$ is an invertible sheaf, it is coherent and 
$\mathcal{Z}$ can be covered by open sets such that  $\mathcal{O}_{\mathcal{Z}}(K)\otimes2\mathcal{I}_{\tilde{\mathscr{F}}}|_{U}$ is a 
free $\mathcal{O}_{\mathcal{Z}}|_{U}$-module.
Since $\varpi:\mathcal{Z}\to \mathcal{U}$ is a flat, proper morphism,  $\mathcal{O}_{\mathcal{Z}}(K)\otimes2\mathcal{I}_{\tilde{\mathscr{F}}}$ is flat over all $t\in\mathcal{U}$. 
Thus, by the theorem 4, we have 
$h^{1}(Z_{t},(\mathcal{O}_{\mathcal{Z}}(K)\otimes2\mathcal{I}_{\tilde{\mathscr{F}}})_{t})$ is an upper semi-continuous function of $t\in\mathcal{U}$
and if $h^{1}(Z_{t}, (\mathcal{O}_{\mathcal{Z}}(K)\otimes2\mathcal{I}_{\tilde{\mathscr{F}}})_{t})$ is constant, 
then $R^{1}\varpi_{*}(\mathcal{O}_{\mathcal{Z}}(K)\otimes2\mathcal{I}_{\tilde{\mathscr{F}}})$ is locally free. 

In Leray spectral sequence, we have 
\[E^{p,q}_{2}=H^{p}(\mathcal{U}, R^{q}\varpi_{*}(\mathcal{O}_{\mathcal{Z}}(K)\otimes2\mathcal{I}_{\tilde{\mathscr{F}}})),\]
and
\[d_{2}: E_{2}^{p,q}\to E_{2}^{p+2, q-1}.\]
By Cartan's Theorem B, $H^{i}(U, R^{q}\varpi_{*}(\mathcal{O}_{\mathcal{Z}}(K)\otimes2\mathcal{I}_{\tilde{\mathscr{F}}})=0$ for $i>0$.
Thus, we get $d_{2}=0$ for all $q$. Therefore, Leray spectral sequence degenerates at $E_{2}$-level. 

Then, we have 
\[E_{2}^{p,q}=E_{\infty}^{p,q}=Gr_{p}H^{p+q}(\mathcal{Z}, \mathcal{O}_{\mathcal{Z}}(K)\otimes2\mathcal{I}_{\tilde{\mathscr{F}}}),\]
where $Gr_{p}$ is the $p$-th filtration of $H^{p+q}(\mathcal{Z}, \mathcal{O}_{\mathcal{Z}}(K)\otimes2\mathcal{I}_{\tilde{\mathscr{F}}})$.
In particular, for $p=0, q=1$, we get
\[E_{2}^{0,1}=Gr_{0}H^{1}(\mathcal{Z}, \mathcal{O}_{\mathcal{Z}}(K)\otimes2\mathcal{I}_{\tilde{\mathscr{F}}})
=H^{1}(\mathcal{Z}, \mathcal{O}_{\mathcal{Z}}(K)\otimes2\mathcal{I}_{\tilde{\mathscr{F}}}).\]
Thus, we get 
\[H^{0}(\mathcal{U}, R^{1}\varpi_{*}(\mathcal{O}_{\mathcal{Z}}(K)\otimes2\mathcal{I}_{\tilde{\mathscr{F}}}))
=H^{1}(\mathcal{Z}, \mathcal{O}_{\mathcal{Z}}(K)\otimes2\mathcal{I}_{\tilde{\mathscr{F}}}).\]
By Theorem 4, if $h^{1}(Z_{t}, (\mathcal{O}_{\mathcal{Z}}(K)\otimes2\mathcal{I}_{\tilde{\mathscr{F}}})_{t})$ is constant, then 
the following restriction map
\[r_{t}:H^{1}(\mathcal{Z}, \mathcal{O}_{\mathcal{Z}}(K)\otimes2\mathcal{I}_{\tilde{\mathscr{F}}})\to
 H^{1}(Z_{t}, (\mathcal{O}_{\mathcal{Z}}(K)\otimes2\mathcal{I}_{\tilde{\mathscr{F}}})_{t})\]
is surjective. 

\vspace{20pt}

Let $\omega $ be a self-dual harmonic 2-form on a smooth, oriented, compact Riemannian manifold. 
Then 
\[0=\int_{M}<dd^{*}+d^{*}d\omega, \omega>=\int_{M}<d\omega, d\omega>+\int_{M}<d^{*}\omega, d^{*}\omega>.\]
Thus, $d\omega=0$. 
Therefore, $H^{1}(Z_{t}, \mathcal{O}_{Z_{t}}(K_{t}))$ corresponds to the space of self-dual harmonic 2-forms on $K3\#3\overline{\mathbb{CP}_{2}}$, 
which has dimension 3.

\begin{Theorem}
Let $\varpi:\mathcal{Z}\to \mathcal{U} $, $\mathcal{U}\subset \mathbb{C}$, be a 1-parameter family of standard deformation of $Z_{0}$ such that each fiber $Z_{t}$
correspond to the twistor space of $(K3\#3\overline{\mathbb{CP}_{2}}, g_{t})$ constructed in [6]. 
Then for $t\neq 0$, $h^{1}(Z_{t}, \mathcal{O}_{Z_{t}}(K_{t}))=3$.
If $h^{1}(Z_{0}, (\mathcal{O}_{\mathcal{Z}}(K)\otimes2\mathcal{I}_{\tilde{\mathscr{F}}}|_{0})=3$, 
then a given real element  $H^{1}(Z_{0}, (\mathcal{O}_{\mathcal{Z}}(K)\otimes2\mathcal{I}_{\tilde{\mathscr{F}}})|_{0})$
can be extended nearby fiber so that we get a real element  in $H^{1}(Z_{t}, \mathcal{O}_{Z_{t}}(K_{t}))$ for $t\neq 0$.
\end{Theorem}

\vspace{50pt}

\section{\large\textbf{First Cohomology of the singular fiber and a nondegenerate element}}\label{S:Intro}
In this section, we show that $h^{1}(Z_{0}, (\mathcal{O}_{\mathcal{Z}}(K)\otimes2\mathcal{I}_{\tilde{\mathscr{F}}})_{0})=3$
and there is a nondegenerate element in $H^{1}(Z_{0}, (\mathcal{O}_{\mathcal{Z}}(K)\otimes2\mathcal{I}_{\tilde{\mathscr{F}}})_{0})$. 
Then using Theorem 5, we prove Theorem 1. 

Let $\tilde{\mathscr{F}}=\tilde{F_{1}}\cup\tilde{F_{2}}\cup\tilde{F_{3}}$ and 
$\mathcal{I}_{\tilde{\mathscr{F}}}=\mathcal{I}_{\tilde{F_{1}}}+\mathcal{I}_{\tilde{F_{2}}}+\mathcal{I}_{\tilde{F_{3}}}$.

\newtheorem{Lemma}{Lemma}
\begin{Lemma}
1. $[-\tilde{\mathscr{F}}]|_{\tilde{F}_{i}}=[Q_{i}]$.

2.  $[-\tilde{\mathscr{F}}]|_{\tilde{Z}}=[-\mathscr{Q}]$.

\end{Lemma}
\begin{proof}
First, we assume that the first factor of $Q_{i}=\mathbb{CP}_{1}\times\mathbb{CP}_{1}$ is the twistor line and the second factor is the blown up line
so that the normal bundle of $Q_{i}$ in $\tilde{Z}$ is $\mathcal{O}(1,-1)$ and in $\tilde{F}_{i}$ is $\mathcal{O}(-1, 1)$. 
Note that by adjunction formula, $K_{Z_{t}}=K_{\mathcal{Z}}\otimes[Z_{t}]|_{Z_{t}}$. But the normal bundle of $Z_{t}$ for $t\neq 0$ in $\mathcal{Z}$ is trivial.
Thus, $K_{Z_{t}}=K_{\mathcal{Z}}|_{Z{t}}$ for $t\neq 0$. We also note that normal bundle of $\tilde{F}_{i}$ restricted to $\tilde{F}_{i}-Q_{i}$ is trivial. 

We consider the following. 
\[0\longrightarrow V_{Q_{1},\tilde{F_{1}}}\longrightarrow V_{Q_{1}}\longrightarrow V_{\tilde{F_{1}}}|_{Q_{1}}\longrightarrow 0\]
Here $V_{Q_{1},\tilde{F_{1}}}$ is the normal bundle of $Q_{1}$ in $\tilde{F_{1}}$, which is $\mathcal{O}(-1, 1)$
 and $V_{Q_{1}}$ is the normal bundle of $Q_{1}$ in $\mathcal{Z}$, which is $\mathcal{O}(1, -1)\oplus \mathcal{O}(-1, 1)$
and $V_{\tilde{F_{1}}}|_{Q_{1}}$ is the normal bundle of $\tilde{F_{1}}$ in $\mathcal{Z}$ restricted to $Q_{1}$.
From this, we get that $V_{\tilde{F_{1}}}|_{Q_{1}}$ is $\mathcal{O}(1, -1)$.
Note that $Q_{1}$ can be seen as a divisor of $\tilde{F}_{1}$. 
Then, $[-Q_{1}]|_{Q_{1}}=\mathcal{O}(1, -1)$ and $[-Q_{1}]|_{\tilde{F_{1}}-Q}$ is trivial, 
where $[\cdot]$ denotes the line bundle corresponding to a divisor. 
Thus, $[-Q_{1}]$ is the same with the normal bundle of $\tilde{F}_{1}\subset\mathcal{Z}$. 

Similarly, we can show  $[-\tilde{\mathscr{F}}]|_{\tilde{Z}}=[-\mathscr{Q}]$.
We can check the result does not depend on the choice of the factor of $Q_{i}=\mathbb{CP}_{1}\times\mathbb{CP}_{1}$.
\end{proof}

\vspace{20pt}

The same proof with Lemma 1 implies that 
\[[\tilde{Z}]|_{\tilde{Z}}=[-\mathscr{Q}] \hspace{5pt} and \hspace{5pt}  [\tilde{F_{i}}]|_{\tilde{F_{i}}}=[-Q_{i}].\]
Thus, we have
\[K_{\mathcal{Z}}|_{\tilde{Z}}=K_{\tilde{Z}}\otimes [-\tilde{Z}]|_{\tilde{Z}}=K_{\tilde{Z}}\otimes [\mathscr{Q}].\]

\begin{Lemma}

1. $(\mathcal{O}_{\mathcal{Z}}(K_{\mathcal{Z}})\otimes2\mathcal{I}_{\tilde{\mathscr{F}}})_{t}=
(\mathcal{O}_{\mathcal{Z}}(K_{\mathcal{Z}})\otimes2\mathcal{I}_{\tilde{\mathscr{F}}})|_{Z_{t}}\otimes\mathcal{O}_{Z_{t}}=\mathcal{O}_{Z_{t}}(K_{t})$ for $t\neq 0$. 

2.$(\mathcal{O}_{\mathcal{Z}}(K_{\mathcal{Z}})\otimes2\mathcal{I}_{\tilde{\mathscr{F}}})|_{\tilde{Z}}\otimes\mathcal{O}_{\tilde{Z}}
=\mathcal{O}_{\tilde{Z}}(K_{\tilde{Z}}\otimes [-\mathscr{Q}])$

3.$(\mathcal{O}_{\mathcal{Z}}(K_{\mathcal{Z}})\otimes2\mathcal{I}_{\tilde{\mathscr{F}}})|_{\tilde{F}_{i}}\otimes\mathcal{O}_{\tilde{F}_{i}}
=\mathcal{O}_{\tilde{F_{i}}}(K_{\tilde{F}_{i}}\otimes 3[Q_{i}])$

\end{Lemma}

\begin{Lemma}
Let $\pi:\tilde{S}\to S$ be the blowing up along a submanifold $W$ with codimension $k+1$ and let $\mathcal{I}$ be the ideal sheaf of $\tilde{W}$. 
Then we have $\mathcal{O}_{\tilde{S}}(\pi^{*}K_{S})=\mathcal{O}_{\tilde{S}}(K_{\tilde{S}})\otimes\mathcal{I}^{k}_{\tilde{W}}$.
\end{Lemma}

\begin{Remark}
From Lemma 3, we get 
\[\mathcal{O}_{\tilde{Z}}(K_{\tilde{Z}}\otimes [-\mathscr{Q}])=\mathcal{O}_{\tilde{Z}}(\pi^{*}K_{Z}).\]

\end{Remark}

\begin{Lemma}
$\mathcal{O}_{\tilde{Z}}(K_{\tilde{Z}}\otimes [-\mathscr{Q}])$ is non-trivial along the twistor line direction and trivial along the blown up direction 
of $Q_{i}=\mathbb{CP}_{1}\times\mathbb{CP}_{1}$. 
On the other hand, $\mathcal{O}_{\tilde{F}_{i}}(K_{\tilde{F_{i}}}\otimes [3Q_{i}])$ is non-trivial along the blown up direction
and trivial along the twistor line direction.

\end{Lemma}
\begin{proof}
Suppose we have chosen the factor of $Q_{i}$ so that the normal bundle of $Q_{i}$ in $\tilde{Z}$ is $\mathcal{O}(1, -1)$
and in $\tilde{F}_{i}$ is $\mathcal{O}(-1,1)$. Then the first factor of $Q_{i}$ is the twistor line direction in $\tilde{Z}$
and the blown up direction in $\tilde{F}_{i}$. 

Using this, we get
\[K_{\tilde{Z}}|_{Q_{i}}=K_{Q_{i}}\otimes[-Q_{i}]|_{Q_{i}}=\mathcal{O}(-2, -2)\otimes\mathcal{O}(-1, 1)=\mathcal{O}(-3, -1),\]
\[K_{\tilde{Z}}\otimes[-\mathscr{Q}]|_{Q_{i}}=K_{\tilde{Z}}|_{Q_{i}}\otimes [-\mathscr{Q]}|_{Q_{i}}
=\mathcal{O}(-3, -1)\otimes\mathcal{O}(-1, 1)=\mathcal{O}(-4, 0).\]

Similarly, we have 
\[K_{\tilde{F}_{i}}|_{Q_{i}}=K_{Q_{i}}\otimes[-Q_{i}]|_{Q_{i}}=\mathcal{O}(-2, -2)\otimes\mathcal{O}(1, -1)=\mathcal{O}(-1, -3).\]
\[K_{\tilde{F}_{i}}\otimes3[Q_{i}]|_{Q_{i}}=\mathcal{O}(-1, -3)\otimes\mathcal{O}(-3, 3)=\mathcal{O}(-4, 0).\]

\vspace{20pt}

\end{proof}
Below, we state K\"unneth formula in oder to calculate the cohomology $H^{i}(\mathbb{CP}_{1}\times\mathbb{CP}_{1}, \mathcal{O}(a, b))$. 
\begin{Theorem}
[26] Let $\mathcal{F}, \mathcal{G}$ be cohernet sheaves on $X$ and $Y$ respectively, which are projective varieties over a filed $k$.
Let $\pi_{1}$ is the projection map from $X\times Y$ to $X$ and Similarly, $\pi_{2}$ to $Y$.  
Then the following holds.
\[H^{m}(X\times Y, \pi^{*}_{1}\mathcal{F}\otimes_{\mathcal{O}_{X\times Y}}\pi^{*}_{2}\mathcal{G})\cong\ \oplus_{p+q=m} H^{p}(X, \mathcal{F})\otimes_{k} H^{q}(Y, \mathcal{G}).\]

\end{Theorem}

Our goal is to show that $H^{1}(Z_{0},(\mathcal{O}_{\mathcal{Z}}(K_{\mathcal{Z}})\otimes2\mathcal{I}_{\tilde{\mathscr{F}}})_{0})$ 
is 3-dimensional and there is a nondegenerate element in 
$H^{1}(Z_{0}, (\mathcal{O}_{\mathcal{Z}}(K_{\mathcal{Z}})\otimes2\mathcal{I}_{\tilde{\mathscr{F}}})_{0})$.
The reason to choose $\mathcal{O}_{\mathcal{Z}}(K)\otimes2\mathcal{I}_{\tilde{\mathscr{F}}}$ instead of 
$\mathcal{O}_{\mathcal{Z}}(K)\otimes\mathcal{I}_{\tilde{\mathscr{F}}}$ is that we would like to get an element 
$\alpha\in H^{1}(Z_{0},(\mathcal{O}_{\mathcal{Z}}(K_{\mathcal{Z}})\otimes2\mathcal{I}_{\tilde{\mathscr{F}}})_{0})$, such that $\alpha|_{Q}$ is nonzero. 
Let $i: Q_{i}\to\tilde{Z}$ and $i: Q_{i}\to\tilde{F_{i}}$ be inclusion maps. Then we note that 
\[h^{1}(Q_{i}, i^{*}\mathcal{O}_{\tilde{Z}}(\pi^{*}K_{Z}))=h^{1}(Q_{i}, \mathcal{O}_{Q_{i}}(-4, 0))=3.\]
If we use $K\otimes\mathcal{I}_{\tilde{\mathscr{F}}}$, we get 
\[H^{1}(Q_{i}, i^{*}(\mathcal{O}_{\tilde{Z}}(\pi^{*}K_{Z}\otimes [Q_{i}])))=H^{1}(Q_{i}, \mathcal{O}_{Q_{i}}(-3, -1))=0.\]

\vspace{20pt}

\begin{Lemma}
Let $Z$ be the twistor space of $K3$-surface with a Ricci-flat K\"ahler metric and $\tilde{Z}$ be the blown up of $Z$ along three twistor lines. 
Let $F$ be the twistor space of $\overline{\mathbb{CP}_{2}}$ with Fubini-Study metric and $\tilde{F}$ be the blown up of $F$ along a twistor line. 
Then we have 
\[h^{1}(\tilde{Z}, \mathcal{O}_{\tilde{Z}}(\pi_{1}^{*}K_{Z}))=3\]
\[h^{1}(\tilde{F}, \mathcal{O}_{\tilde{F}}(\pi_{2}^{*}K_{F}))=0.\]
\end{Lemma}
\begin{proof}
Below we show that 
$h^{1}(\tilde{Z}, \mathcal{O}_{\tilde{Z}}(\pi_{1}^{*}K_{Z}))=h^{1}(Z, \mathcal{O}_{Z}(K_{Z}))$ and
$h^{1}(\tilde{F}, \mathcal{O}_{\tilde{F}}(\pi_{2}^{*}K_{F}))=h^{1}(F, \mathcal{O}_{F}(K_{F}))$.
$H^{1}(Z, \mathcal{O}_{Z}(K_{Z}))$ corresponds to the space of self-dual harmonic 2-forms  on $K3$-surface. 
Since $b_{+}=3$ on this surface, we get $h^{1}(Z, \mathcal{O}_{Z}(K_{Z}))=3$.
Similarly, since there is no self-dual harmonic 2-form on $\overline{\mathbb{CP}_{2}}$, we have $h^{1}(F, \mathcal{O}_{F}(K_{F}))=0$.

\end{proof}

\begin{Lemma} 
Let $f: X \to Y $ be a continuous map of topological spaces and let $\mathcal{G}$ be a sheaf of abelian groups on $X$. 
If $R^{i}f_{*}\mathcal{G}=0$ for $i>0$, then for all $i\geq 0$, there is a following isomorphism. 
\[H^{i}(X, \mathcal{G})\cong H^{i}(Y, f_{*}\mathcal{G}).\]
\end{Lemma}
\begin{proof}
This follows from the Leray spectral sequence argument. 
We refer to ([13], Proposition 3.0.6) for details of the proof. 
\end{proof}

Therefore, in order to prove  Lemma 5, we need to prove $R^{i}\pi_{*}\mathcal{O}_{\tilde{Z}}(\pi^{*}K)=0$ for $i>0$
and $\pi_{*}\mathcal{O}_{\tilde{Z}}(\pi^{*}K)=K$. For this, we use following Propositions ([28] V2. p.124, [13] Proposition 3.0.8)
\newtheorem{Proposition}{Proposition}
\begin{Proposition}
Let $X$ and $Y$ be complex manifolds and suppose $f:X\to Y$ be a holomorphic proper and submersive map and $\mathcal{G}$ be a coherent analytic sheaf on $X$. 
If $H^{i}(f^{-1}(y), \mathcal{G}|_{f^{-1}(y)})=0$ for all $y\in Y$, then $R^{i}f_{*}(\mathcal{G})=0$. 
\end{Proposition}

\begin{Proposition}
Let $\pi:(\tilde{Z}, Q)\to (Z, l)$ be  blowing up of a twistor line $l$ and $K$ be the canonical bundle on $Z$. 
Then, we have $R^{i}\pi_{*}\mathcal{O}_{\tilde{Z}}(\pi^{*}K)=0$. 
\end{Proposition}
\begin{proof}
This follows from $\pi^{-1}(y)$ is a point or $\mathbb{P}^{1}$ and $H^{1}(\mathbb{P}^{1}, \mathcal{O})=0$. 
\end{proof}

Then we get 
\[H^{i}(\tilde{Z}, \mathcal{O}_{\tilde{Z}}(\pi^{*}K))=H^{i}(Z, \pi_{*}\mathcal{O}_{\tilde{Z}}(\pi^{*}K)).\]
Using the Projection formula and Zariski's Main Theorem [10], we get $\pi_{*}\mathcal{O}_{\tilde{Z}}(\pi^{*}K)=\mathcal{O}_{Z}(K)$
([13], Lemma 3.0.9, 3.0.10).

\begin{Lemma}
(Projection formula)
Let $f:(X, \mathcal{O}_{X})\to (Y, \mathcal{O}_{Y})$ be a morphism of ringed spaces. 
If $\mathcal{G}$ is $\mathcal{O}_{X}$-module and $\mathcal{E}$  is locally free $\mathcal{O}_{Y}$-module of finite rank, then 
\[f_{*}(\mathcal{G}\otimes_{\mathcal{O}_{X}}f^{*}\mathcal{E})=f_{*}\mathcal{G}\otimes_{\mathcal{O}_{Y}}\mathcal{E}.\]
For $\mathcal{G}=\mathcal{O}_{X}$, we get
\[f_{*}f^{*}\mathcal{E}=f_{*}\mathcal{O}_{X}\otimes_{\mathcal{O}_{Y}}\mathcal{E}.\]

\end{Lemma}

\begin{Lemma}
(Zariski's Main Theorem, weak version)
Let $X$ and $Y$ be noetherian integral schemes and let $f:X\to Y$ be a birational projective morphism. 
If $Y$ is normal, then $f_{*}\mathcal{O}_{X}=\mathcal{O}_{Y}$. 
\end{Lemma}

\vspace{50pt}
First, we consider $H^{1}(\tilde{Z}, \mathcal{O}_{\tilde{Z}}(\pi^{*}K_{Z}))$. 
\begin{Lemma}
Let $(M, g)$ be an oriented, smooth, compact 4-dimensional Riemannian manifold 
and $g$ has the property $s=W_{+}=0$, where $s$ is the scalar curvature of $g$. 
Then any self-dual harmonic 2-form on $M$ is parallel. \end{Lemma}
\begin{proof}
From the following the Weitzenb\"ock formula for self-dual 2-forms, we have
\[\Delta\omega=\nabla^{*}\nabla\omega-W_{+}(\omega, \cdot)+\frac{s}{3}\omega.\]
Thus, if $s=W_{+}=0$ and $M$ is compact, we get $\nabla\omega=0$ for a self-dual harmonic 2-form.
\end{proof}

By Yau's theorem, a K3 surface admits a Ricci-flat K\"ahler metric [30]. 
Note that for a K\"ahler metric, the self-dual Weyl tensor $W_{+}$ is determined by the scalar curvature $s$.
Namely, $W_{+}$ takes the following form in a K\"ahler case. 
\[W_{+}=\left(
\begin{matrix}
-\frac{s}{12}&0&0\\
0&-\frac{s}{12}&0\\
0&0&\frac{s}{6}\\
\end{matrix}
\right)
\]
Thus, $s=0$ if and only if $W_{+}=0$. 
Therefore, K3 surface with Ricci-flat K\"ahler metric has $W_{+}=0$. 
In particular, a self-dual harmonic 2-form on K3 surface with Ricci-flat K\"ahler metric is parallel.

\begin{Lemma}
[5] Let $(Y, g)$ be a smooth, oriented Riemannian n-manifold, $n\geq 2$ and let $P$ be a point of $Y$. 
Let $\phi$ be a differential $l$-form on $Y-p$ such that $d\phi=0$ and $d*\phi=0$. 
If there is a neighborhood $U$ of $p$ and a positive constant $C$ such that $|\phi|<C$ on $U-p$, 
then $\phi$ extends to $Y$ uniquely and smoothly and $d\phi=0$ and $d*\phi=0$ on $Y$.
\end{Lemma}

\begin{Remark}
We note that if we assume $\phi$ is self-dual on $Y-\{p\}$ in the Lemma 10, 
the extended $\phi$ is also self-dual. 
Let $*$ be the Hodge-star operator of $(Y, g)$. 
Then 
\[*\phi(p)=lim_{z\to p}*\phi(z)=lim_{z\to p}\phi(z)=\phi(p)\]
since $*\phi$ and $\phi$ are smooth sections of $\Lambda^{2}$ and $\phi$ is self-dual on $Y-\{p\}$.

\end{Remark}
\vspace{20pt}

\begin{Lemma}
Let $\alpha\in H^{1}(\tilde{Z}, \mathcal{O}_{\tilde{Z}}(\pi^{*}K_{Z}))$ be a real element and suppose that $\alpha$ is not identically zero. 
Then $\alpha|_{l}\neq 0$ for any real twistor line $l\in \tilde{Z}-\mathscr{Q}$ and $\alpha|_{Q_{i}}\in H^{1}(Q_{i}, \mathcal{O}_{Q_{i}}(-4,0))$ is not zero. 
\end{Lemma}
\begin{proof}
Since $[Q]$ is trivial on $\tilde{Z}-Q$, by restricting cohomology on the subset, we get 
$\alpha|_{\tilde{Z}-\mathscr{Q}}\in H^{1}(\tilde{Z}-\mathscr{Q}, \mathcal{O}_{\tilde{Z}-\mathscr{Q}}(K_{\tilde{Z}-\mathscr{Q}}))
=H^{1}(Z-\mathcal{L}, \mathcal{O}_{Z-\mathcal{L}}(K_{Z-\mathcal{L}}))$, 
where $\mathcal{L}=l_{1}\cup l_{2}\cup l_{3}$ and it is real. 
Thus, by Theorem 4, $\alpha|_{\tilde{Z}-\mathscr{Q}}$ corresponds to a closed real self-dual 2-form $\phi$ on $K3-\{p_{1}, p_{2}, p_{3}\}$.
We claim $\phi$ is bounded near $p_{i}$. 
So it is enough to show that $\alpha|_{l}\in H^{1}(\mathbb{CP}_{1}, \mathcal{O}(-4))$ is bounded,
where $l$ is a twistor line on $Z$ near $l_{i}$. 
Since $\pi^{*}K_{Z}$ is non-trivial along the twistor line direction, $\alpha|_{l_{i}\times \{z\}}$ is bounded for any $z\in \mathbb{CP}_{1}$. 
Thus, for a twistor line $l$ near $l_{i}$, $\alpha|_{l}$ is bounded. 
Thus, $\phi$ is bounded near $p_{i}$. 
Then by the Lemma 10 and Remark 3, $\phi$ is extended smoothly as a self-dual 2-form and $d\phi=0$ on $K3$.
In particular, it is harmonic. 
By the Lemma 9, $\phi$ is parallel on $K3$. Then $||\phi||$ is constant 
and at a point $q$, there exists an orthonormal basis such that $\phi(q)=e_{1}\wedge e_{2}+e_{3}\wedge e_{4}$.
Then, it can be easily checked that $\phi$ is nondegenerate at $q$. 
\end{proof}

As a Corollary of Lemma 11, we get the following. 
\newtheorem{Corollary}{Corollary}
\begin{Corollary}
Let $i:Q\to \tilde{Z}$ be the inclusion map. Then the restriction map 
$r_{1}: H^{1}(\tilde{Z}, \mathcal{O}_{\tilde{Z}}(\pi^{*}K_{Z}))\to H^{1}(Q, i^{*}(\mathcal{O}_{\tilde{Z}}(\pi^{*}K_{Z})))$ is an isomorphism. 
\end{Corollary}
\begin{proof}
Note that $h^{1}(\tilde{Z}, \mathcal{O}_{\tilde{Z}}(\pi^{*}K_{Z})) =h^{1}(Q,  i^{*}(\mathcal{O}_{\tilde{Z}}(\pi^{*}K_{Z})))=3$ and $r_{1}$ is injective by Lemma 11. 
Thus, $r_{1}$ is an isomorphism. 
\end{proof}

\vspace{50pt}
We need to calculate $h^{1}(Z_{0}, (\mathcal{O}_{\mathcal{Z}}(K_{\mathcal{Z}})\otimes2\mathcal{I}_{\tilde{\mathscr{F}}})_{0})$.
First, we consider one of $(\tilde{F}_{i}, Q_{i})$, which we denote by $(\tilde{F}, Q)$. 
Let us consider the following exact sequence. 
\[0\longrightarrow \mathcal{O}_{\tilde{F}}(\pi^{*}K_{F})\longrightarrow   \mathcal{O}_{\tilde{F}}(K_{\tilde{F}}) \longrightarrow 
\mathcal{O}_{Q}(K_{\tilde{F}}|_{Q})\longrightarrow 0.\]
Using the fact $H^{0}(Q, \mathcal{O}_{Q}(K_{\tilde{F}}|_{Q}))=H^{1}(Q, \mathcal{O}_{Q}(K_{\tilde{F}}|_{Q}))=0$, we get 
\[H^{1}(\tilde{F}, \mathcal{O}_{\tilde{F}}(\pi^{*}K_{F}))=H^{1}(F, \mathcal{O}_{F}(K_{F}))=0.\]

Similarly, we get $H^{1}(\tilde{F}, \mathcal{O}_{\tilde{F}}(K_{\tilde{F}}\otimes n[Q]))=0$ for $n=1, 2$. 
From the exact sequence below, 
\[0\longrightarrow \mathcal{O}_{\tilde{F}}(K_{\tilde{F}}\otimes 2[Q]) \longrightarrow \mathcal{O}_{\tilde{F}}(K_{\tilde{F}}\otimes 3[Q]) \longrightarrow 
\mathcal{O}_{Q}(K_{\tilde{F}}\otimes 3[Q]|_{Q})\longrightarrow 0,\]
we get 
\[0\longrightarrow H^{1}(\tilde{F}, \mathcal{O}_{\tilde{F}}(K_{\tilde{F}}\otimes 3[Q]))
\xrightarrow{r_{2}} H^{1}(Q, \mathcal{O}_{Q}(K_{\tilde{F}}\otimes 3[Q]|_{Q}))\]
\[\longrightarrow H^{2}(\tilde{F}, \mathcal{O}_{\tilde{F}}(K_{\tilde{F}}\otimes 2[Q]))\longrightarrow H^{2}(\tilde{F}, \mathcal{O}_{\tilde{F}}(K_{\tilde{F}}\otimes 3[Q]))\longrightarrow 0.\]

In order to calculate $H^{2}(\tilde{F}, \mathcal{O}_{\tilde{F}}(K_{\tilde{F}}\otimes 2[Q]))$, we describe the twistor space $F$. 
Let $V$ be the vector space which is isomorphic to $\mathbb{C}^{3}$
and $V^{*}$ is the dual vector space of $V$. 
Then $F$ is given by [2], [6], [18].
\[\{([v], [w])\in P(V)\times P(V^{*})\cong \mathbb{CP}_{2}\times\mathbb{CP}_{2}|v\cdot w=0\}.\]
Thus, the twistor space $F$ is a hypersurface of $\mathbb{CP}_{2}\times\mathbb{CP}_{2}$
given by a linear system $\mathcal{O}(1,1)$.

\begin{Lemma}
Let $F$ be the twistor space of $\overline{\mathbb{CP}_{2}}$ with Fubini-Study metric and opposite orientation to usual one, which is a flag manifold. 
Then we have $H^{i}(F, \mathcal{O}(K_{F}))=0$ for $i=0, 1, 2.$ and $h^{3}(F, \mathcal{O}_{F}(K_{F}))=1$.
\end{Lemma}
\begin{proof}
Let $P:=\mathbb{CP}_{2}\times\mathbb{CP}_{2}$.
By adjunction formula, we have 
\[K_{F}=K_{P}\otimes\mathcal{O}(1,1)|_{F}\]
and $K_{P}=\mathcal{O}_{P}(-3, -3)$. Thus, in particular, we get $K_{F}=\mathcal{O}_{F}(-2, -2)$.

We consider the following exact sequence. 
\[0\longrightarrow \mathcal{O}_{P}(K_{P})\longrightarrow \mathcal{O}_{P}(K_{P}\otimes \mathcal{O}(1,1))\longrightarrow \mathcal{O}_{F}(K_{P}\otimes\mathcal{O}(1,1)|_{F})\longrightarrow 0.\]
Then from this, we get the following long exact sequence,
\[0\longrightarrow H^{0}(P, \mathcal{O}(-3,-3))\longrightarrow H^{0}(P,\mathcal{O}(-2,-2))\longrightarrow H^{0}(F, \mathcal{O}_{F}(K_{F}))\longrightarrow\]
\[\longrightarrow H^{1}(P, \mathcal{O}(-3,-3))\longrightarrow H^{1}(P,\mathcal{O}(-2,-2))\longrightarrow H^{1}(F, \mathcal{O}_{F}(K_{F})\longrightarrow\]
\[\longrightarrow H^{2}(P, \mathcal{O}(-3,-3))\longrightarrow H^{2}(P,\mathcal{O}(-2,-2))\longrightarrow H^{2}(F, \mathcal{O}_{F}(K_{F})\longrightarrow\]
\[\longrightarrow H^{3}(P, \mathcal{O}(-3,-3))\longrightarrow H^{3}(P,\mathcal{O}(-2,-2))\longrightarrow H^{3}(F, \mathcal{O}_{F}(K_{F}))\longrightarrow H^{4}(P,  \mathcal{O}(-3,-3))\longrightarrow 0.\]
Note that all terms are zero using the cohomology of $\mathbb{CP}_{2}$ and K\"unneth formula except the last two terms. 
Since $H^{4}(\mathbb{CP}_{2}\times\mathbb{CP}_{2}, \mathcal{O}(-3, -3))=H^{2}(\mathbb{CP}_{2}, \mathcal{O}(-3))\otimes H^{2}(\mathbb{CP}_{2}, \mathcal{O}(-3))$.
Thus, $h^{4}(P, \mathcal{O}_{P}(-3, -3))=1$ and therefore, we get 
\[h^{3}(F, \mathcal{O}_{F}(K_{F}))=1.\]
\end{proof}

\vspace{50pt}

Again from the following short exact sequence
\[0\longrightarrow \mathcal{O}_{\tilde{F}}(\pi^{*}K_{F})\longrightarrow  \mathcal{O}_{\tilde{F}}(K_{\tilde{F}}) \longrightarrow \mathcal{O}_{Q}(K_{\tilde{F}}|_{Q})\longrightarrow 0,\]
we get
\[0\longrightarrow H^{2}(\tilde{F}, \mathcal{O}_{\tilde{F}}(\pi^{*}K_{F_{i}}))\longrightarrow H^{2}(\tilde{F}, \mathcal{O}_{\tilde{F}}(K_{\tilde{F}}))\longrightarrow 0\]
\[ \longrightarrow H^{3}(\tilde{F}, \mathcal{O}_{\tilde{F}}(\pi^{*}K_{F}))\longrightarrow H^{3}(\tilde{F}, \mathcal{O}_{\tilde{F}}(K_{\tilde{F}}))\longrightarrow 0.\]
Using $H^{2}(\tilde{F}, \mathcal{O}_{\tilde{F}}(\pi^{*}K_{F}))=H^{2}(F, \mathcal{O}_{F}(K_{F}))=0$, we get $H^{2}(\tilde{F}, \mathcal{O}_{\tilde{F}}(K_{\tilde{F}}))=0$. 
Since $h^{3}(\tilde{F}, \mathcal{O}_{\tilde{F}}(\pi^{*}K_{F}))=h^{3}(F, \mathcal{O}_{F}(K_{F}))=1$, we get $h^{3}(\tilde{F}, \mathcal{O}_{F}(K_{\tilde{F}}))=1$. 

The following short exact sequence
\[0\longrightarrow  \mathcal{O}_{\tilde{F}}(K_{\tilde{F}}\otimes [Q])\longrightarrow  \mathcal{O}_{\tilde{F}}(K_{\tilde{F}}\otimes 2[Q]) \longrightarrow 
\mathcal{O}_{Q}(K_{\tilde{F}}\otimes 2[Q]|_{Q})
\longrightarrow 0,\]
gives 
\[0\longrightarrow H^{2}(\tilde{F},  \mathcal{O}_{\tilde{F}}(K_{\tilde{F}}\otimes [Q]))\longrightarrow H^{2}(\tilde{F}, \mathcal{O}_{\tilde{F}}(K_{\tilde{F}}\otimes 2[Q]))\longrightarrow\]
\[\longrightarrow 0 \longrightarrow H^{3}(\tilde{F},  \mathcal{O}_{\tilde{F}}(K_{\tilde{F}}\otimes [Q]))
\longrightarrow H^{3}(\tilde{F},  \mathcal{O}_{\tilde{F}}(K_{\tilde{F}}\otimes 2[Q]))\longrightarrow 0.\]
Thus, we get $ H^{2}(\tilde{F},  \mathcal{O}_{\tilde{F}}(K_{\tilde{F}}\otimes [Q]))\cong H^{2}(\tilde{F},  \mathcal{O}_{\tilde{F}}(K_{\tilde{F}}\otimes 2[Q]))$.

Also from the following short exact sequence
\[0\longrightarrow \mathcal{O}_{\tilde{F}}(K_{\tilde{F}})\longrightarrow  \mathcal{O}_{\tilde{F}}(K_{\tilde{F}}\otimes [Q])\longrightarrow \mathcal{O}_{Q}(K_{\tilde{F}}\otimes [Q]|_{Q})\longrightarrow 0,\]
we get 
\[0\longrightarrow H^{2}(\tilde{F}, \mathcal{O}_{\tilde{F}}(K_{\tilde{F}}\otimes [Q]))\longrightarrow
H^{2}(Q, \mathcal{O}_{Q}(-2, -2)) \longrightarrow \]
\[\longrightarrow H^{3}(\tilde{F}, \mathcal{O}_{\tilde{F}}(K_{\tilde{F}}))\longrightarrow 
H^{3}(\tilde{F}, \mathcal{O}_{\tilde{F}}(K_{\tilde{F}}\otimes [Q]))\longrightarrow 0.\]
Note that  $h^{2}(Q, \mathcal{O}_{Q}(-2, -2))=1$ and $h^{2}(\tilde{F}, \mathcal{O}_{\tilde{F}}(K_{\tilde{F}}))=0$. 
Thus, we get $h^{2}(\tilde{F}, \mathcal{O}_{\tilde{F}}(K_{\tilde{F}}\otimes [Q]))=0$ or $1$. 
If $H^{2}(\tilde{F}, \mathcal{O}_{\tilde{F}}(K_{\tilde{F}}\otimes [Q]))=0$, then $H^{2}(\tilde{F}, \mathcal{O}_{\tilde{F}}(K_{\tilde{F}}\otimes 2[Q]))=0$
and therefore in this case, $r_{2}:H^{1}(\tilde{F}, \mathcal{O}_{\tilde{F}}(K_{\tilde{F}}\otimes 3[Q]))\to H^{1}(Q, \mathcal{O}_{Q}(K_{\tilde{F}}\otimes 3[Q]) |_{Q})$ is an isomorphism. 
If $h^{2}(\tilde{F}, \mathcal{O}_{\tilde{F}}(K_{\tilde{F}}\otimes [Q]))=1$, then $h^{2}(\tilde{F}, \mathcal{O}_{\tilde{F}}(K_{\tilde{F}}\otimes 2[Q]))=1$.

From the following short exact sequence
\[0\longrightarrow \mathcal{O}_{\tilde{F}}(K_{\tilde{F}}\otimes 2[Q])\longrightarrow \mathcal{O}_{\tilde{F}}(K_{\tilde{F}}\otimes 3[Q])\longrightarrow 
\mathcal{O}_{Q}(K_{\tilde{F}}\otimes 3[Q]|_{Q})
\longrightarrow 0,\]
we get
\[0\longrightarrow H^{1}(\tilde{F}, \mathcal{O}_{\tilde{F}}(K_{\tilde{F}}\otimes 3[Q]))\xrightarrow{r_{2}}H^{1}(Q, \mathcal{O}_{Q}(K_{\tilde{F}}\otimes 3[Q]|_{Q}))\longrightarrow\]
\[\longrightarrow H^{2}(\tilde{F}, \mathcal{O}_{\tilde{F}}(K_{\tilde{F}}\otimes 2[Q]))\longrightarrow H^{2}(\tilde{F}, \mathcal{O}_{\tilde{F}}(K_{\tilde{F}}\otimes 3[Q]))\longrightarrow 0.\]

Thus, if $H^{2}(\tilde{F}, \mathcal{O}_{\tilde{F}}(K_{\tilde{F}}\otimes [Q]))\neq 0$ and $H^{2}(\tilde{F}, \mathcal{O}_{\tilde{F}}(K_{\tilde{F}}\otimes 3[Q]))\neq 0$, 
then $h^{2}(\tilde{F}, \mathcal{O}_{\tilde{F}}(K_{\tilde{F}}\otimes [Q]))=h^{2}(\tilde{F}, \mathcal{O}_{\tilde{F}}(K_{\tilde{F}}\otimes 3[Q]))=1$
and in this case, $r_{2}$ is an isomorphism. 
If $H^{2}(\tilde{F}, \mathcal{O}_{\tilde{F}}(K_{\tilde{F}}\otimes [Q]))\neq 0$ and $H^{2}(\tilde{F}, \mathcal{O}_{\tilde{F}}(K_{\tilde{F}}\otimes 3[Q]))=0$, 
then $h^{1}(\tilde{F}, \mathcal{O}_{\tilde{F}}(K_{\tilde{F}}\otimes 3[Q]))=2$ and $r_{2}$ is injective. 

Thus, we can conclude that either $h^{1}(\tilde{F},  \mathcal{O}_{\tilde{F}}(K_{\tilde{F}}\otimes 3[Q]))=h^{1}(Q, \mathcal{O}_{Q}(K_{\tilde{F}}\otimes 3[Q]|_{Q}))=
h^{1}(\mathbb{CP}_{1}\times\mathbb{CP}_{1}, \mathcal{O}(-4, 0))=3$ and $r_{2}$ is an isomorphism, or
$h^{1}(\tilde{F}, \mathcal{O}_{\tilde{F}}(K_{\tilde{F}}\otimes 3[Q]))=2$ and $r_{2}$ is injective.

\begin{Remark}
Note that if $H^{2}(\tilde{F}, \mathcal{O}_{\tilde{F}}(K_{\tilde{F}}\otimes [Q]))=0$, 
then $H^{2}(\tilde{F}, \mathcal{O}_{\tilde{F}}(K_{\tilde{F}}\otimes 2[Q]))=H^{2}(\tilde{F}, \mathcal{O}_{\tilde{F}}(K_{\tilde{F}}\otimes 3[Q]))=0$.
In this case, $h^{1}(\tilde{F}, \mathcal{O}_{\tilde{F}}(K_{\tilde{F}}\otimes 3[Q]))=3$ and $r_{2}$ is an isomorphism.

\end{Remark}

\begin{Lemma}
Let $F$ be the twistor space of $\overline{\mathbb{CP}}_{2}$ with Fubini-Study metric with nonstandard orientation
and let $\pi: (\tilde{F}, Q)\to (F, l)$ be the blowing up along a twistor line $l\subset F$. 
Let $r_{2} :  H^{1}(\tilde{F}, \mathcal{O}_{\tilde{F}}(K_{\tilde{F}}\otimes 3[Q]))\to H^{1}(Q, \mathcal{O}_{Q}(K_{\tilde{F}}\otimes 3[Q]|_{Q}))$ be the restriction map.
Then either $h^{1}(\tilde{F},  \mathcal{O}_{\tilde{F}}(K_{\tilde{F}}\otimes 3[Q]))=3$ and $r_{2}$ is an isomorphism or 
$h^{1}(\tilde{F}, \mathcal{O}_{\tilde{F}}(K_{\tilde{F}}\otimes 3[Q]))=2$ and $r_{2}$ is injective.

\end{Lemma}

\vspace{20pt}

\begin{Lemma}

 $h^{1}(Z_{0}, \mathcal{O}_{\mathcal{Z}}(K\otimes 2\mathcal{I}_{\tilde{\mathscr{F}}})_{0})\geq3$
\end{Lemma}
\begin{proof}
Consider $\varpi:\mathcal{Z}\to \mathcal{U}$ be the standard complex deformation. 
Each fiber $Z_{t}$ for $t\neq 0$ is the twistor space of $(K3\#3\overline{\mathbb{CP}}_{2}, g_{t})$, 
where $g_{t}$ is a family of ASD metrics constructed in [6]. 
Note that $(\mathcal{O}_{\mathcal{Z}}(K)\otimes 2\mathcal{I}_{\tilde{\mathscr{F}}})_{t}=K_{Z_{t}}$ for $t\neq 0$. 
By Lemma 5, we have $h^{1}(Z_{t}, \mathcal{O}_{Z_{t}}(K_{Z_{t}}))=3$. By upper semicontinuity,   
we get dim $h^{1}(Z_{0}, \mathcal{O}_{\mathcal{Z}}(K\otimes 2\mathcal{I}_{\tilde{\mathscr{F}}})_{0})\geq3$.

\end{proof}

\begin{Proposition}
$h^{1}(Z_{0}, \mathcal{O}_{\mathcal{Z}}(K\otimes 2\mathcal{I}_{\tilde{\mathscr{F}}})_{0})=3$.

\end{Proposition}

\begin{proof}
Let $a:\tilde{Z_{1}}\amalg\tilde{\mathscr{F}}\to Z_{0}$ be the normalization map. We consider the following exact sequence. 
\[0\longrightarrow \mathcal{O}_{\mathcal{Z}}(K\otimes 2\mathcal{I}_{\tilde{\mathscr{F}}})_{0}\longrightarrow a_{*}(\mathcal{O}_{\tilde{Z}}(K_{\tilde{Z}}\otimes[-\mathscr{Q}])
\oplus\mathcal{O}_{\tilde{\mathscr{F}}}(K_{\tilde{\mathscr{F}}}\otimes[3\mathscr{Q}]))
\longrightarrow\mathcal{O}_{\mathscr{Q}}(\pi^{*}K_{Z}|_{\mathscr{Q}})\longrightarrow 0.\]

Note that each restriction of $\mathcal{O}_{\tilde{Z}}(K_{\tilde{Z}}\otimes[-\mathscr{Q}])$ and 
$\mathcal{O}_{\tilde{\mathscr{F}}}(K_{\tilde{\mathscr{F}}}\otimes[3\mathscr{Q}])$ to $\mathscr{Q}$ are the same.

Note that $H^{0}(Q_{i}, \mathcal{O}_{Q_{i}}(\pi^{*}K_{Z}|_{Q_{i}}))=H^{0}(Q_{i}, \mathcal{O}_{Q_{i}}(-4, 0))=0$. 
Therefore, we get the following long exact sequence
\[0 \longrightarrow H^{1}(Z_{0}, \mathcal{O}_{\mathcal{Z}}(K\otimes 2\mathcal{I}_{\tilde{\mathscr{F}}})_{0})
\xrightarrow{f} H^{1}(\tilde{Z}, \mathcal{O}_{\tilde{Z}}(K_{\tilde{Z}}\otimes[-\mathscr{Q}]))\oplus 
H^{1}(\tilde{\mathscr{F}}, \mathcal{O}_{\tilde{\mathscr{F}}}(K_{\tilde{\mathscr{F}}}\otimes[3\mathscr{Q}])\longrightarrow\]
\[\xrightarrow{r} H^{1}(\mathscr{Q}, \mathscr{O}_{Q}(\pi^{*}K_{Z}|_{\mathscr{Q}}))\longrightarrow 
H^{2}(Z_{0}, \mathcal{O}_{\mathcal{Z}}(K\otimes 2\mathcal{I}_{\tilde{\mathscr{F}}})_{0})\longrightarrow\]
\[\longrightarrow H^{2}(\tilde{Z},  \mathcal{O}_{\tilde{Z}}(K_{\tilde{Z}}\otimes[-\mathscr{Q}])\oplus 
H^{2}(\tilde{\mathscr{F}}, \mathcal{O}_{\tilde{\mathscr{F}}}(K_{\tilde{\mathscr{F}}}\otimes[3\mathscr{Q}]))\longrightarrow\cdot\cdot\cdot\]
Since $h^{1}(Q_{i}, \mathcal{O}_{Q_{i}}(\pi^{*}K_{Z}|_{Q_{i}}))=h^{1}(Q_{i}, \mathcal{O}_{Q_{i}}(-4, 0))=3$, 
we have $h^{1}(\mathscr{Q}, \mathcal{O}_{\mathscr{Q}}(\pi^{*}K_{Z}|_{\mathscr{Q}}))=9$. 
The map $r$ is given by 
\[r=(\alpha|_{Q_{1}}-\beta_{1}|_{Q_{1}}, \alpha|_{Q_{2}}-\beta_{2}|_{Q_{2}}, \alpha|_{Q_{3}}-\beta_{3}|_{Q_{3}}),\]
where $\alpha\in H^{1}(\tilde{Z}, \mathcal{O}_{\tilde{Z}}(\pi^{*}K_{Z}))$ and $\beta_{i}\in H^{1}(\tilde{F_{i}}, \mathcal{O}_{\tilde{F}}(K_{\tilde{F}_{i}}\otimes 3[Q_{i}]))$. 
Since $r_{2}: H^{1}(\tilde{F_{i}},  \mathcal{O}_{\tilde{F}}(K_{\tilde{F}_{i}}\otimes 3[Q_{i}]))\xrightarrow{r_{2}}H^{1}(Q_{i}, \mathcal{O}_{Q_{i}}(K_{\tilde{F}_{i}}\otimes 3[Q_{i}])|_{Q_{i}})$
 is injective, 
we get dim Im r $\geq$ $h^{1}(\tilde{\mathscr{F}}, \mathcal{O}_{\tilde{\mathscr{F}}}(K_{\tilde{\mathscr{F}}}\otimes[3\mathscr{Q}]))$. Thus, we get 
dim ker r $\leq$ dim $H^{1}(\tilde{Z}, \mathcal{O}_{\tilde{Z}}(K_{\tilde{Z}}\otimes[-\mathscr{Q}]))$, which is 3. 
Since $f$ is injective and dim $H^{1}(Z_{0}, \mathcal{O}_{\mathcal{Z}}(K\otimes 2\mathcal{I}_{\tilde{\mathscr{F}}})_{0})\geq3$, 
we get dim Im $f\geq 3$. Thus, we get $3\leq$ dim Im $f$ $=dim $ ker $r\leq 3$. Thus, dim Im $f=3$ and therefore, 
dim $H^{1}(Z_{0}, \mathcal{O}_{\mathcal{Z}}(K\otimes 2\mathcal{I}_{\tilde{\mathscr{F}}})_{0})=3$.

\end{proof}

\begin{Corollary}
The restriction map $r_{2}:H^{1}(\tilde{F}, \mathcal{O}_{\tilde{F}}(K_{\tilde{F}}\otimes[3Q]))\to H^{1}(Q, \mathcal{O}_{Q}(K_{\tilde{F}}\otimes 3[Q]|_{Q}))$
is an isomorphism.

\end{Corollary}
\begin{proof}
Suppose $r_{2}$ is not an isomorphism. Then 
$h^{1}(\tilde{\mathscr{F}}, \mathcal{O}_{\tilde{\mathscr{F}}}(K_{\tilde{\mathscr{F}}}\otimes[-\mathscr{Q}]))<9$.
In this case, Im  r is strictly greater than $H^{1}(\tilde{\mathscr{F}}, K\otimes2\mathcal{I}_{\tilde{\mathscr{F}}}|_{\tilde{\mathscr{F}}})$
because $r_{1}$ is an isomorphism by Corollary 1.
Then dim Im $f$= dim Ker $r$ $<3=$ dim $H^{1}(\tilde{Z}, \mathcal{O}_{\tilde{Z}}(K_{\tilde{Z}}\otimes[-\mathscr{Q}]))$, 
which is a contradiction since dim Im $f=3$ from Proposition 3. 
Thus, $r_{2}$ is an isomorphism.

\end{proof}

From the argument before Remark 4 and Corollary 2, we get the following corollary. 

\begin{Corollary}
Either $H^{2}(\tilde{F}, \mathcal{O}_{\tilde{F}}(K_{\tilde{F}}\otimes n[Q]))=0$ or $h^{2}(\tilde{F}, \mathcal{O}_{\tilde{F}}(K_{\tilde{F}}\otimes n[Q]))=1$ for $n=1, 2, 3$.
\end{Corollary}

\vspace{50pt}

We show that there is a nondegenerate element in $H^{1}(Z_{0}, (\mathcal{O}_{\mathcal{Z}}(K)\otimes 2\mathcal{I}_{\tilde{\mathscr{F}}})_{0})$.
By lemma 11, if $\alpha\in H^{1}(\tilde{Z}, \pi^{*}K_{Z})$ is not identically zero, there $\alpha|_{Q}\neq0$ and $\alpha|_{l}\neq 0$ for every twistor line 
$l\subset \tilde{Z}-Q$. Below, we prove the similar one for a cohomology class in $H^{1}(\tilde{F}, \mathcal{O}_{\tilde{F}}(K_{\tilde{F}}\otimes3[Q]))$.

\begin{Lemma}
Let $\beta$ be a real element of $H^{1}(\tilde{F}, \mathcal{O}_{\tilde{F}}(K_{\tilde{F}}\otimes3[Q]))$
and suppose $\beta$ is not identically zero. 
Then $\beta|_{l}\neq 0$ for any  real twistor line $l\in \tilde{F}-Q$ and $\beta|_{Q}\neq 0$. 
\end{Lemma}
\begin{proof}
Since $[Q]$ is trivial on $\tilde{F}-Q$, 
by restricting a real cohomology $\beta\in H^{1}(\tilde{F}, \mathcal{O}_{\tilde{F}}(K_{\tilde{F}}\otimes3[Q]))$ to $\tilde{F}-Q$, 
we get a real element in $H^{1}(\tilde{F}-Q, K_{\tilde{F}-Q})$. 
By Theorem 4, this element corresponds to a real self-dual harmonic 2-form on $(\overline{\mathbb{CP}}_{2}-\{y\}, g_{FS})$, 
where $g_{FS}$ is the restriction of Fubini-Study metric. Moreover, $g_{FS}$ on $\overline{\mathbb{CP}}_{2}-\{y\}$
with the non-standard orientation is conformal to Burns metric.
By Corollary 2, the restriction map $r_{2}:H^{1}(\tilde{F}, \mathcal{O}_{\tilde{F}}(K_{\tilde{F}}\otimes3[Q]))\to H^{1}(Q, \mathcal{O}_{Q}(K_{\tilde{F}}\otimes3[Q]|_{Q}))$
is an isomorphism. Note that $H^{1}(Q, \mathcal{O}_{Q}(K_{\tilde{F}}\otimes3[Q]|_{Q}))=H^{1}(\mathbb{CP}_{1}\times \mathbb{CP}_{1}, \mathcal{O}(0, -4))$. 
Thus, on $Q$, there is a rational curve $\mathbb{CP}_{1}$ such that the restriction of $K_{\tilde{F}}\otimes 3[Q]|_{Q}$ on it is $\mathcal{O}(-4)$
and $H^{1}(\mathbb{CP}_{1}, \mathcal{O}(-4))\neq 0$. On $\tilde{F}-Q$, $K_{\tilde{F}}\otimes 3[Q)|_{Q}$ is $K_{\tilde{F}-Q}$, 
and for a twistor line $l$ on $\tilde{F}-Q$, which is $\mathbb{CP}_{1}$, $K_{F}|_{l}=\mathcal{O}(-4)$ [2]. 
Namely, the restriction of the sheaf are the same for the rational curve on $Q$ and any twistor line on  $\tilde{F}-Q$.
From this, we get $r:H^{1}(\tilde{F}, \mathcal{O}_{\tilde{F}}(K_{\tilde{F}}\otimes3[Q]))\to H^{1}(\mathbb{CP}_{1}, \mathcal{O}(-4))$
is an isomorphism for any twistor line on  $\tilde{F}-Q$.

\end{proof}

\begin{Proposition}
There is a real element $\gamma\in H^{1}(Z_{0}, (\mathcal{O}_{\mathcal{Z}}(K)\otimes 2\mathcal{I}_{\tilde{\mathscr{F}}})_{0})$ such that 
$\gamma|_{l}\neq 0$, for any real twistor line in $\tilde{Z}-\mathscr{Q}$ and $\tilde{\mathscr{F}}-\mathscr{Q}$. 

\end{Proposition}
\begin{proof}

From the description of $H^{1}(Z_{0}, (\mathcal{O}_{\mathcal{Z}}(K)\otimes 2\mathcal{I}_{\tilde{\mathscr{F}}})_{0})$ in the following long exact sequence, 
\[0 \longrightarrow H^{1}(Z_{0}, (\mathcal{O}_{\mathcal{Z}}(K)\otimes 2\mathcal{I}_{\tilde{\mathscr{F}}})_{0})
\xrightarrow{f} H^{1}(\tilde{Z}, \mathcal{O}_{\tilde{Z}}(K_{\tilde{Z}}\otimes[-\mathscr{Q}]))\oplus 
H^{1}(\tilde{\mathscr{F}}, \mathcal{O}_{\tilde{\mathscr{F}}}(K_{\tilde{\mathscr{F}}}\otimes[3\mathscr{Q}]))\longrightarrow\]
\[\xrightarrow{r} H^{1}(\mathscr{Q}, \mathcal{O}_{\mathscr{Q}}(\pi^{*}K_{Z}|_{\mathscr{Q}}))\longrightarrow,\]
an element of $H^{1}(Z_{0}, (\mathcal{O}_{\mathcal{Z}}(K)\otimes 2\mathcal{I}_{\tilde{\mathscr{F}}})_{0})$ is given by the kernel of the map 
 $r=(\alpha|_{Q_{i}}-\beta_{i}|_{Q_{i}})$, where $\alpha\in H^{1}(\tilde{Z}, \mathcal{O}_{\tilde{Z}}(K_{\tilde{Z}}\otimes[-\mathscr{Q}]))$
and $\beta_{i}\in H^{1}(\tilde{F_{i}}, \mathcal{O}_{\tilde{F}_{i}}(K_{\tilde{F}_{i}}\otimes[3Q_{i}]))$. 
Since Ker $r=\mathbb{C}^{3}$, we take $(\alpha, \beta_{i})\in$ Ker $r$, which is real and not identically zero.

First, we assume that $\alpha$ is not identically zero. 
Then by Lemma 11, $\alpha_{l}\neq 0$ for any twistor line $l$ in $\tilde{Z}-\mathscr{Q}$ and $\alpha|_{Q_{i}}\neq 0$ for any $i$. 
Then we have $\alpha|_{Q_{i}}=\beta_{i}|_{Q_{i}}\neq 0$. 
By Lemma 15, $\beta_{i}|_{l}\neq 0$ for any twistor line in $\tilde{F_{i}}-Q_{i}$. 
If we assume $\beta_{j}$ is not identically zero for some $j$, then by Lemma 15, $\beta_{j}|_{l}\neq 0$ for any twistor line in $\tilde{F}_{j}-Q_{j}$
and $\beta_{j}|_{Q_{j}}\neq 0$. Then $\alpha|_{Q_{j}}\neq 0$. By the same argument using Lemma 11 and 15, we get 
$\alpha|_{l}\neq 0$ for any twistor line $l\in \tilde{Z}-\mathscr{Q}$ and $\beta_{i}|_{l}\neq 0$ for any twistor line $l\in \tilde{F}_{i}-Q_{i}$ for any $i$.

\end{proof}

We have shown that $h^{1}(Z_{0}, (\mathcal{O}_{\mathcal{Z}}(K)\otimes 2\mathcal{I}_{\tilde{\mathscr{F}}})_{0})=3$ for all $t$ including the singular fiber. 
Thus, by Theorem 5, a given element in $H^{1}(Z_{0}, (\mathcal{O}_{\mathcal{Z}}(K)\otimes 2\mathcal{I}_{\tilde{\mathscr{F}}})_{0})$ can be 
extended to nearby fiber. If we take an element given in Proposition 4, then we get a nondegenerate real element of $H^{1}(Z_{t}, \mathcal{O}_{Z_{t}}(K_{t}))$ for $t$ near $0$
since nondegeneracy is an open condition. 
Thus, we get a nondegenerate self-dual harmonic 2-form on $(K3\#3\overline{\mathbb{CP}_{2}}, g_{t})$
and therefore an almost-K\"ahler anti-self-dual metric in the conformal class of $g_{t}$. 
Let $\alpha\in H^{1}(\tilde{Z}, \pi^{*}K_{Z})$ and $\mathscr{Q}=\cup_{1\leq i\leq n} Q_{i}$ for any $n\geq 4$.
A self-dual harmonic 2-form corresponding to $\alpha|_{\tilde{Z}-\mathscr{Q}}$ can be extended to $K3$ by the argument of Lemma 11.
The same argument is easily extended to cover cases $K3\#n\overline{\mathbb{CP}_{2}}$ for $n\geq 4$. 
Since $K3\#3\overline{\mathbb{CP}_{2}}$ does not admit a scalar-flat K\"ahler metric, we get a strictly alnmost-K\"ahler anti-self-dual metric 
on $K3\#3\overline{\mathbb{CP}_{2}}$ for $n\geq 3$. This finishes the proof of Theorem 1.

\vspace{50pt}

\section{\large\textbf{Scalar curvatures of almost-K\"ahler anti-self-dual metrics }}\label{S:Intro}

Recall that the anti-self-duality is a conformal invariant. 
By the solution of Yamabe problem [22], each conformal class on a compact manifold of dim $\geq 3$ has a representative whose scalar curvature is constant. 
There are three types according to the sign of the scalar curvature. 
It is interesting to note the type of almost-K\"ahler anti-self-dual metrics on $K3\#n\overline{\mathbb{CP}_{2}}$ for $n\geq 3$.

From the Weitzenb\"ock formula for a self-dual 2-form, 
\[\Delta\omega=\nabla^{*}\nabla\omega-2W_{+}(\omega, \cdot)+\frac{s}{3}\omega,\]
if an anti-self-dual compact manifold $(M, g)$ is positive type, then $b_{+}(M)=0$ (Corollary 1, [27]).
Otherwise, there is a self-dual harmonic 2-form with respect to an an anti-self-dual metric on $M$ with constant positive scalar curvature. 
Then we have 
\[0=<\nabla^{*}\nabla\omega, \omega>+\frac{s}{3}<\omega, \omega>.\]
Then we get $\int_{M}\frac{s}{3}<\omega, \omega>\leq 0$, which is a contradiction. 
Since $K3\#n\overline{\mathbb{CP}_{2}}$ for $n\geq 3$ has self-dual harmonic 2-forms, 
an anti-self-dual conformal class on $K3\#n\overline{\mathbb{CP}_{2}}$ for $n\geq 3$ cannot be positive type. 

Moreover, from the following result in [19, Proposition 3.5], $K3\#n\overline{\mathbb{CP}_{2}}$ for $n\geq 3$ cannot admit an anti-self-dual metric with zero scalar curvature.

\begin{Proposition}
[19] Suppose $M$ be a smooth, oriented, compact four-dimensional manifold. 
If $M$ admits a scalar-flat anti-self-dual metric, then $M$ is homeomorphic to $k\overline{\mathbb{CP}_{2}}$ for $k\geq 5$ 
or $M$ is diffeomorphic to  $\mathbb{CP}_{2}\#n\overline{\mathbb{CP}_{2}}$ for $n\geq 10$, or diffeomorphic to $K3$ surface. 
\end{Proposition}
Thus, we can conclude that an anti-self-dual conformal class  on $K3\#3\overline{\mathbb{CP}_{2}}$ for $n\geq 3$ is negative type. 
The following result was proven in case $b_{1}(M)=0$ in [27] and in general in [9]. 
\begin{Theorem}
[9, 27] Suppose $(M, c)$ be a compact, oriented anti-self-dual conformal manifold and its conformal class
contains a metric of constant negative scalar curvature. 
Then the corresponding twistor space does not have a nontrivial divisor. 

\end{Theorem}
From this, we get the twistor space $Z$ of $(K3\#n\overline{\mathbb{CP}_{2}}, g)$ for $n\geq 3$ does not admit a nontrivial divisor,
where $g$ is an anti-self-dual metric with negative type. 
This is Corollary 2 in [27].

\vspace{20pt}

Note that the property of $K3$-surface we need in this paper in the construction of almost-K\"ahler anti-self-dual metrics is that
$b_{+}\neq 0$ and the metric is anti-self-dual and has vanishing scalar curvature. 
Then among the list given in Proposition 5, $\mathbb{CP}_{2}\#n\overline{\mathbb{CP}_{2}}$ for $n\geq 10$ with scalar-flat K\"ahler metrics have these properties. 
Thus, instead of $K3$ surface, we may use $\mathbb{CP}_{2}\#n\overline{\mathbb{CP}_{2}}$ for $n\geq 10$. 

We note that it was shown that $\mathbb{CP}_{2}\#n\overline{\mathbb{CP}_{2}}$ for $n\geq14$ admit scalar-flat K\"ahler metrics
by twistor method [15, 16].
The optimal case $\mathbb{CP}_{2}\#n\overline{\mathbb{CP}_{2}}$ for $n=10$ was successful using gluing method in [25]. 

\begin{Theorem}
[20] Suppose $M$ be a compact scalar-flat K\"ahler surface such that $c_{1}\neq 0$. 
Let $Z$ be its twistor space and $D$ be the corresponding divisor. 
Suppose $M$ be not a minimal ruled surface of genus $\gamma\geq 2$ such that $H^{0}(M, \Theta_{M})\neq 0$ 
and $M$ be not of the form $\mathbb{P}(L\oplus \mathcal{O})\to S_{\gamma}$, where $S_{\gamma}$ is a riemann surface of genus $\gamma\geq 2$. 
Then $H^{2}(Z, \Theta_{Z}\otimes\mathcal{I}_{D})=0$. 

\end{Theorem}

\begin{Proposition}
[20, 14] Suppose $M$ be a compact scalar-flat K\"ahler surface such that $c_{1}\neq 0$. 
Let $Z$ be its twistor space with the corresponding divisor $D$. 
If  $H^{2}(Z, \Theta_{Z}\otimes\mathcal{I}_{D})=0$, then $H^{2}(Z, \Theta_{Z})=0$.

\end{Proposition}
Let $g$ be a scalar-flat K\"ahler metric on $\mathbb{CP}_{2}\#n\overline{\mathbb{CP}_{2}}$ for $n\geq10$.
Then by Theorem 8 and Proposition 6, the twistor space of $(\mathbb{CP}_{2}\#n\overline{\mathbb{CP}_{2}}, g)$
has $H^{2}(Z, \Theta_{Z})=0$. Thus, Donaldson-Friedman construction can be applied to the pair 
$(\mathbb{CP}_{2}\#n\overline{\mathbb{CP}_{2}}, g)$ for $n\geq10$ and $(\overline{\mathbb{CP}_{2}}, g_{FS})$.
Moreveor, our construction of nondegenerate self-dual harmonic 2-form also applies in these cases. 
The existence of strictly almost-K\"ahler anti-self-dual metrics on $\mathbb{CP}_{2}\#n\overline{\mathbb{CP}_{2}}$ for $n\geq 11$
 is already shown by deforming scalar-flat K\"ahler metrics [14]. 
The method of showing existence of almost-K\"ahler anti-self-dual metrics on such manifolds in this paper is different from this case.
On the other hand, since $\mathbb{CP}_{2}\#n\overline{\mathbb{CP}_{2}}$ admit scalar-flat K\"ahler metrics unlike $K3\#n\overline{\mathbb{CP}_{2}}$, 
we can only state the theorem in the following way.

\begin{Theorem}
There is an almost-K\"ahler anti-self-dual metric on $\mathbb{CP}_{2}\#n\overline{\mathbb{CP}_{2}}$ for $n\geq 11$. 
\end{Theorem}

\vspace{50pt}

\section{\large\textbf{Appendix: Calculation of the second cohomology of the singular fiber}}\label{S:Intro}

In this section, we consider $H^{2}(Z_{0}, (\mathcal{O}_{\mathcal{Z}}(K)\otimes 2\mathcal{I}_{\tilde{\mathcal{F}}})_{0})$,
where $Z_{0}=\tilde{Z}\cup_{\mathscr{Q}}\tilde{\mathscr{F}}$ is obtained from the twistor space of K3 surface with Ricci-flat K\"ahler metric 
and 3 copies of twistor space of $(\overline{\mathbb{CP}_{2}}, g_{FS})$.

By Serre Duality 
\[H^{2}(Z_{t}, K_{Z_{t}})=H^{1}(Z_{t}, \mathcal{O})^{*}.\]
It was shown that [8], [11], [17] that $H^{1}(Z_{t}, \mathcal{O})$ corresponds to the first cohomology group of the following complex
\[\Lambda^{0}\xrightarrow{d}\Lambda^{1}\xrightarrow{d_{+}}\Lambda_{+}^{2},\]
where $\Lambda^{i}$ are complex-valued $i$-forms on $M$ and $d_{+}\omega$ is the self-dual part of $d\omega$ of a 1-form $\omega$.
From this, we get next useful result. We recall the proof briefly following Corollary 3.2 in [8].

\begin{Theorem}
[8, 11, 17] For the twistor space $Z$ of a compact, smooth, oriented riemannian 4-manifold with an anti-self-dual metric $(M, g)$, 
$H^{1}(Z, \mathcal{O})=H^{1}(M, \mathbb{C})$.
\end{Theorem}
\begin{proof}
By the above argument, it suffices to show that if $d_{+}\omega=0$, then $d\omega=0$. 
Define $\alpha:=d\omega$. Then $*\alpha=-\alpha$ by definition of $d_{+}$ and $d_{+}\omega=0$. 
Then we have 
\[||\alpha||^{2}=\int_{M}\alpha\wedge *\alpha=-\int_{M}\alpha\wedge\alpha=-\int_{M}d\omega\wedge\alpha=-\int_{M}\omega\wedge d\alpha=0.\]

\end{proof}

\begin{Lemma}
For a twistor space $Z$ of $K3$ surface with Ricci-flat K\"ahler metric, we have $H^{2}(Z, \mathcal{O}_{Z}(K_{Z}))=0$. 
For $Z_{t}$, which is a twistor space of $(K3\#n\overline{\mathbb{CP}_{2}}, g_{t})$,
where $g_{t}$ is a family of anti-self-dual metrics constructed in [6], we have  $H^{2}(Z_{t}, \mathcal{O}_{Z_{t}}(K_{Z_{t}}))=0$. 
\end{Lemma}
\begin{proof}
Since $K3$ surface and $K3\#n\overline{\mathbb{CP}_{2}}$ are simply connected, 
we get immediately the conclusion from Theorem 10 and Serre Duality. 
\end{proof}

In Corollary 3, it is shown that either $H^{2}(\tilde{F}, \mathcal{O}_{\tilde{F}}(K_{\tilde{F}}\otimes n[Q]))=0$ or 
$h^{2}(\tilde{F}, \mathcal{O}_{\tilde{F}}(K_{\tilde{F}}\otimes n[Q]))=1$ for $n=1, 2, 3$.
We claim if $H^{2}(\tilde{F}, \mathcal{O}_{\tilde{F}}(K_{\tilde{F}}\otimes n[Q]))=0$, then $H^{2}(Z_{0}, (\mathcal{O}_{\mathcal{Z}}(K)\otimes 2\mathcal{I}_{\tilde{\mathscr{F}}})_{0})=0$
and if $h^{2}(\tilde{F}, \mathcal{O}_{\tilde{F}}(K_{\tilde{F}}\otimes n[Q]))=1$, 
then $h^{2}(Z_{0}, (\mathcal{O}_{\mathcal{Z}}(K)\otimes 2\mathcal{I}_{\tilde{\mathscr{F}}})_{0})=3$

\begin{Proposition}
If $H^{2}(\tilde{F}, \mathcal{O}_{\tilde{F}}(K_{\tilde{F}}\otimes [Q])) =0$, then $H^{2}(Z_{0}, (\mathcal{O}_{\mathcal{Z}}(K)\otimes 2\mathcal{I}_{\tilde{\mathscr{F}}})_{0})=0$. 
\end{Proposition}
\begin{proof}
We consider again the long exact sequence
\[0 \longrightarrow H^{1}(Z_{0}, (\mathcal{O}_{\mathcal{Z}}(K)\otimes 2\mathcal{I}_{\tilde{\mathscr{F}}})_{0})
\xrightarrow{f} H^{1}(\tilde{Z}, \mathcal{O}_{\tilde{Z}}(K_{\tilde{Z}}\otimes[-\mathscr{Q}]))\oplus
H^{1}(\tilde{\mathscr{F}}, \mathcal{O}_{\tilde{\mathscr{F}}}(K_{\tilde{\mathscr{F}}}\otimes[3\mathscr{Q}]))\longrightarrow\]
\[\xrightarrow{r} H^{1}(\mathscr{Q}, \mathscr{O}_{Q}(\pi^{*}K_{Z}|_{\mathscr{Q}}))\xrightarrow{g}
 H^{2}(Z_{0}, (\mathcal{O}_{\mathcal{Z}}(K)\otimes 2\mathcal{I}_{\tilde{\mathscr{F}}})_{0})\longrightarrow\]
\[\xrightarrow{h} H^{2}(\tilde{Z}, \mathcal{O}_{\tilde{Z}}(K_{\tilde{Z}}\otimes[-\mathscr{Q}]))\oplus
H^{2}(\tilde{\mathscr{F}}, \mathcal{O}_{\tilde{\mathscr{F}}}(K_{\tilde{\mathscr{F}}}\otimes[3\mathscr{Q}]))\longrightarrow0.\]
From Remark 4, if $H^{2}(\tilde{F},  \mathcal{O}_{\tilde{F}}(K_{\tilde{F}}\otimes [Q]))=0$, 
then $H^{2}(\tilde{\mathscr{F}}, \mathcal{O}_{\tilde{\mathscr{F}}}(K_{\tilde{\mathscr{F}}}\otimes[3\mathscr{Q}]))=0$.
Moreover, from Lemma 16, we get $H^{2}(\tilde{Z}, \mathcal{O}_{\tilde{Z}}(K_{\tilde{Z}}\otimes[-\mathscr{Q}]))
=H^{2}(\tilde{Z}, \mathcal{O}_{\tilde{Z}}(\pi^{*}K_{Z}))=H^{2}(Z, \mathcal{O}_{Z}(K_{Z}))=0$. 
From Corollary 2, we get $r$ is surjective. Therefore, we get 
$H^{2}(Z_{0}, (\mathcal{O}_{\mathcal{Z}}(K)\otimes 2\mathcal{I}_{\tilde{\mathscr{F}}})_{0})=0$.
\end{proof}

\begin{Proposition}
If $h^{2}(\tilde{F}, \mathcal{O}_{\tilde{F}}(K_{\tilde{F}}\otimes [Q]))=1$, 
then $h^{2}(Z_{0}, (\mathcal{O}_{\mathcal{Z}}(K)\otimes 2\mathcal{I}_{\tilde{\mathscr{F}}})_{0})=3$.
\end{Proposition}

\begin{proof}
Again we consider the long exact sequence given in the proof of Proposition 7. 
Note that from Corollary 2, we get $r$ is surjective and from Corollary 3, we get $h^{2}(\tilde{F}, \mathcal{O}_{\tilde{F}}(K_{\tilde{F}}\otimes3[Q]))=1$. 
From this, we get $h^{2}(\tilde{\mathscr{F}}, \mathcal{O}_{\tilde{\mathscr{F}}}(K_{\tilde{\mathscr{F}}}\otimes[3\mathscr{Q}]))=3$. 
Thus, we get $h^{2}(Z_{0}, (\mathcal{O}_{\mathcal{Z}}(K)\otimes 2\mathcal{I}_{\tilde{\mathscr{F}}})_{0})=
h^{2}(\tilde{\mathscr{F}}, \mathcal{O}_{\tilde{\mathscr{F}}}(K_{\tilde{\mathscr{F}}}\otimes[3\mathscr{Q}]))=3$.

\end{proof}

\begin{Remark}
From Lemma 16, we have  $H^{2}(Z_{t}, \mathcal{O}_{Z_{t}}(K_{t}))=H^{1}(Z_{t}, \mathcal{O})^{*}=0$. 
Then depending on $h^{2}(\tilde{F}, \mathcal{O}_{\tilde{F}}(K_{\tilde{F}}\otimes[Q]))$, 
 $h^{2}(Z_{0}, (\mathcal{O}_{\mathcal{Z}}(K)\otimes 2\mathcal{I}_{\tilde{\mathscr{F}}})_{0})=0$ or $3$.
Thus, we cannot conclude about $h^{2}(Z_{0}, (\mathcal{O}_{\mathcal{Z}}(K)\otimes 2\mathcal{I}_{\tilde{\mathscr{F}}})_{0})$.

\end{Remark}

\vspace{20pt}

\newpage

\renewcommand{\refname}{Bibliography}

\vspace{20pt}
Department of Mathematics Education, Korea National University of Education

Email address: kiysd@snu.ac.kr


\begin{thebibliography}{9}

\bibitem{A}
M. F. Atiyah,
\emph{Green's functions for self-dual four-manifolds},
in: Mathematical Analysis and Applications, Part A, vol. 7 of Adv. in Math. Suppl. Stud., Academic Press, New York, 1981, pp. 129-158.
\bibitem{AHS}
M. Atiyah, N. Hitchin and I. Singer,
\emph{Self-duality in four dimensional Riemannian geometry},
Proc. Roy. Soc. London Ser. A 362 (1978), 425 -461.
\bibitem{BS}
$C. B\breve{a}nic\breve{a}$ and $O. St\breve{a}n\breve{a}sil\breve{a}$,
\emph{Algebraic methods in the global theory of complex spaces}, 
John Wiley \& Sons, New York (1976).
\bibitem{BPV}
W. Barth, K. Hulek, C. Peters, and A. V. de Ven, 
\emph{Compact Complex Surfaces}, 
2nd ed. 2004, Springer.
\bibitem{BL}
C. Bishop and C. LeBrun, 
\emph {Anti-Self-Dual 4-Manifolds, Quasi-Fuchsian Groups and Almost-K\"ahler Geometry},
arXiv: 1708. 03824 [math DG]; to appear in Comm. An. Geom. 
\bibitem{DF}
S. Donaldson and R. Friedman,
\emph{Connected sums of self-dual manifolds and deformations of singular spaces},
1989 Nonlinearity, 2 (1989), pp. 197-239.
\bibitem{EPW}
M. G. Eastwood, R. Penrose, and R. O. Wells. Jr.,
\emph{Cohomology and Massless Fields}, 
Commun. Math. Phys. 78. 305-351 (1981).
\bibitem{ES}
M. G. Eastwood and M. A. Singer, 
\emph{The Fr\"ohlicher spectral sequence on a twistor space}, 
J. Differ. Geom. 38 (3), 653-669 (1993). 
\bibitem{G}
P. Gauduchon, 
\emph{Weyl structures on a self-dual conformal manifolds}, 
Proc. Symp. Pure Math. 54 (1993) Part 2, pp. 259-270.
\bibitem{Ha}
R. Hartshorne, 
\emph{Algebraic Geometry}, 
Graduate Texts in Mathematics 52. 
\bibitem{H}
N. J. Hitchin, 
\emph{Linear field equations on self-dual spaces}, 
Proc. R. Soc. Lond. A 370, 173-191 (1980). 
\bibitem{H}
N. J. Hitchin, 
\emph{K\"ahlerian twistor spaces}, 
Proc. London Math. Soc. (3) 43 (1981), 133-150.

\bibitem{K}
M. Kalafat, 
\emph{Self-dual Metrics on 4-Manifolds}, 
Thesis, Stony Brook University. 
\bibitem{K}
I. Kim, 
\emph{Almost-K\"ahler anti-self-dual metrics}, 
Ann. Glob. Anal. Geom. 49 (2016), 369 -391.
\bibitem{KP}
J. -S. Kim and M. Pontecorvo, 
\emph{A new method of constructing scalar-flat K\"ahler surfaces}, 
J. Differ. Geom. 41(2), 449-477 (1995).
\bibitem{KLP}
J. -S. Kim, C. LeBrun and M. Pontecorvo, 
\emph{Scalar-flat K\"ahler surfaces of all genera}, 
J. Reine Angew. Math. 486, 69-95 (1997).

\bibitem{L}
C. LeBrun, 
\emph{Twistors, K\"ahler Manifolds, and Bimeromorphic Geometry I}, 
J. Am. Math. Soc. 5 (1992) 289-316.
\bibitem{L}
C. LeBrun, 
\emph{Twistors for Tourists: A Pocket Guide for Algebraic Geometers}, 
Proc. Symp. Pure Math. 62.2 (1997) 361-385.
\bibitem{L}
C. LeBrun,
\emph{Curvature functionals, optimal metrics, and the differential topology of 4-manifolds},
in Differential Faces of Geometry,  Kluwer Academic/Plenum, 2004.
\bibitem{LS}
C. LeBrun and M. Singer, 
\emph{Existence and deformation theory for scalar-flat K\"ahler metrics on compact complex surfaces},
Invent. math. 112, 273-313 (1993).
\bibitem{LS}
C. LeBrun and M. Singer,
\emph{A Kummer-type construction of self-dual 4-manifolds}, 
Math. Ann. 300. 165-180(1994).
\bibitem{LP}
J. Lee and T. Parker, 
\emph{The Yamabe problem}, 
Bull. Am. Math. Soc. 17 (1987), 37-91.
\bibitem{MS}
D. McDuff and D. Salamon,
\emph{Introduction to Symplectic Topology},
Oxford University Press, Oxford (1995).
\bibitem{P}
R. Penrose, 
\emph{Nonlinear gravitons and curved twistor theory}, 
General Relativity and Gravitation 7 (1976), 31-52.
\bibitem{RS}
Y. Rollin and M. A. Singer, 
\emph{Non-minimal scalar-flat K\"ahler surfaces and parabolic stability}, 
Invent. Math. 162(2), 235-270 (2005). 
\bibitem{SW}
J. H. Sampson and G. Washnitzer,
\emph{A K\"unneth formula for coherent algebraic sheaves},
Illinois J. Math., Volume 3, Issue 3 (1959), 389-402. 

\bibitem{V}
M. Ville, 
\emph{Twistor examples of algebraic dimension zero threefolds}, 
Invent. math. 103 (1991) 537-545. 
\bibitem{V}
C. Voisin, 
\emph{Hodge Theory and Complex Algebraic Geometry 1, 2}, 
Cambridge University Press 2003. 

\bibitem{Y}
S. -T. Yau, 
\emph{On the curvature of compact Hermitian manifolds}, 
Invent. Math. 25, 213 -239 (1974)
\bibitem{Y}
S. -T. Yau, 
\emph{On the Ricci curvature of a compact K\"ahler manifold and the complex Monge-Amp$\grave{e}$re equation. I},
Comm. Pure Appl. Math. 31 (1978), 339-411.



\end{thebibliography}
\end{document}